\documentclass{article}
\usepackage {amssymb}
\usepackage {amsfonts}
\usepackage {amsmath}
\usepackage {amsthm}
\usepackage{graphicx}
\usepackage{graphics}
\usepackage {framed}
\usepackage {amsxtra}
\usepackage {enumerate}
\usepackage[margin=1in]{geometry}

\newcommand{\restrict}{\mbox{$\mid\hspace{-1.1mm}\grave{}$}}

\theoremstyle{plain}
\newtheorem{thm}{Theorem}[section]
\newtheorem{cor}[thm]{Corollary} 
\newtheorem{lem}[thm]{Lemma} 
\newtheorem{claim}{Claim}
\newtheorem{prop}[thm]{Proposition}

\theoremstyle{definition}
\newtheorem{defn}{Definition}

\theoremstyle{remark}

\begin{document}

\title{Transpositional sequences and multigraphs}

\author {Alissa Ellis Yazinski \and Raymond R. Fletcher\footnote{alissarachelle@gmail.com and rfletcher210@gmail.com\ \ Virginia State University\ \ Petersburg, VA 23806}  
\and Donald Silberger\footnote{silbergd@newpaltz.edu and DonaldSilberger@gmail.com\ \ State University of New York\ \ New Paltz, NY 12561}}

\maketitle

\centerline{\sf In memory of Eva Ruth Sturgis Silberger, 1962 August 10 -- 2006 January 15}

\begin{abstract} If ${\bf s} := \langle s_0,s_1,\ldots, s_{k-1}\rangle$ is a sequence of length $|{\bf s}|=k$ of permutations on the set $n:=\{0,1,\ldots, n-1\}$ then $\bigcirc{\bf s} := 
s_0\circ s_1\circ\cdots\circ s_{k-1}\in$ Sym$(n)$, and Seq$({\bf s}) := \{{\bf r} := \langle s_{\psi(0)},\ldots s_{\psi(k-1)}\rangle: \psi\in$ Sym$(k)\}$ denotes the set of {\em rearrangements} 
of {\bf s}. Our overall interest is the set Prod$({\bf s}):=\{\bigcirc{\bf r}:{\bf r}\in$ Seq$({\bf s})\}\subseteq$ Sym$(n)$. 

We focus on {\em transpositional} sequences; that is, on those {\bf s}, each of whose terms is a transposition $(x\, y)$. For {\bf u} a transpositional sequence in Sym$(n)$, there is a natural 
correspondence between Seq$({\bf u})$ and its {\em transpositional multigraph} ${\cal T}({\bf u}) := \langle n;E({\bf u})\rangle$ on the vertex set $n$, where the $k$ simple edges $(x\, y)$ 
in the collection $E({\bf u})$ of multiedges in ${\cal T}({\bf u})$ are the $k$ (not obligatorily distinct) terms in the sequence {\bf u}. 

This paper considers antipodal sorts of sequences {\bf s} in Sym$(n)$: 

We call {\bf s} {\em permutationally complete}, aka perm-complete, iff Prod$({\bf s})\in\{$Alt$(n),\,$Sym$(n)\setminus$Alt$(n)\}$, where Alt$(n)\subseteq$ Sym$(n)$ is the alternating subgroup. 
So, Prod$({\bf s})$ is as large as possible if {\bf s} is perm-complete. We provide sufficient criteria for a transpositional sequence {\bf u} in Sym$(n)$ to be perm-complete and also sufficient criteria for {\bf u} not to be perm-complete.

We call {\bf s} {\em conjugacy invariant}, aka CI,  iff the elements in Prod$({\bf s})$ are mutually conjugate. Prod$({\bf s})$ is small if {\bf s} is CI.  We specify the CI transpositional {\bf u} 
in Sym$(n)$. 
\end{abstract}

\section{Introduction} 

Unless specified otherwise, `sequence in $X$' means finite sequence whose terms are elements in the set $X$. For ${\bf f} := \langle x_0,x_1,\ldots,x_k\rangle$ a sequence, $x_i<_{\bf f} x_j$ \ iff \  
$0\le i<j\le k$, and $x_i\le_{\bf f} x_j$  \ iff \ $0\le i\le j\le k$. 

For $n$ a positive integer, Sym$(n)$ and Alt$(n)$ denote respectively the symmetric group and the alternating group on the set $n := \{0,1,\ldots, n-1\}$. When ${\bf s} := 
\langle s_0,s_1,\ldots, s_{k-1}\rangle$ is a {\em permutational sequence}, i.e., a sequence in Sym$(n)$, then $|{\bf s}|=k$ is its {\em length}, and $\bigcirc{\bf s}:=  
s_0\circ s_1\circ\cdots\circ s_{k-1}$ is its compositional product.\footnote{We compose permutations from left to right. That is, when $\{f,g\}\subseteq$ Sym$(n)$ and $x\in n$, then 
$x(f\circ g) = (xf)g = xfg$.} 

Seq$({\bf s})$ denotes the set of sequences that are arrangements of the terms of {\bf s}. That is to say, Seq$({\bf s})$ denotes the set  
$\{{\bf r}: {\bf r} :=\langle s_{\psi(0)},s_{\psi(1)},\ldots, s_{\psi(k-1)}\rangle$, where $\psi\in$ Sym$(k)\}$. Obviously ${\bf r}\in$ Seq$({\bf s}) \Rightarrow |{\bf r|=|s}|$.\vspace{.5em}

Our general subject is the family of sets  Prod$({\bf s}) := \{\bigcirc{\bf r}: {\bf r}\in$ Seq$({\bf s})\}$ for the sequences {\bf s} in Sym$(n)$. However, fully to characterize the family of all such Prod$({\bf s})$ seems daunting. 
So we confine ourselves to the subclass, of that class of {\bf s}, which is treated in the papers \cite{Denes, Eden, Polya, Silberger}, whose results we extend.  

Plainly, either Prod$({\bf s})\subseteq$ Alt$(n)$ or Prod$({\bf s})\subseteq$ Sym$(n)\setminus$Alt$(n)$. Also, $|$Prod$({\bf s})|\le|$Seq$({\bf s})|\le |{\bf s}|!$ \vspace{.5em}  

When $f\in$ Sym$(n)$, the expression supp$(f)$ denotes the set of $x\in n$ for which $xf\not=x$. If {\bf s} is a permutational sequence then Supp$({\bf s})$ denotes the family of 
all supp$(g)$ for which $g$ is a term in {\bf s}.\vspace{.5em} 

By a {\em transposition} we mean a permutation $f\in$ Sym$(n)$ for which there exist elements $a\not=b$ in $n$ with $af=b$, with $bf =a$ and with $xf =x$ for all $x\in n\setminus\{a,b\}$. 
For $n\ge2$, the set of transpositions in Sym$(n)$ is written $1^{n-2}2^1$. By a {\em transpositional sequence} we mean a sequence in the subset $1^{n-2}2^1$ of Sym$(n)$.. 

By the {\em transpositional multigraph} ${\cal T}({\bf u})$ of a sequence {\bf u} in $1^{n-2}2^1$, we mean the labeled multigraph on the vertex set $n$ that has an $(x\, y)$ as a multiedge  
of multiplicity $\mu(g)\ge0$ if and only if the transposition $g := (x\, y)$ occurs exactly $\mu(g)$ times as a term in {\bf u}. For convenience, we will usually take it that ${\cal T}({\bf u})$ is connected, in which event of course    
$\bigcup$Supp$({\bf u})=$ supp$({\bf u})=n$. It is obvious that ${\cal T}({\bf r})={\cal T}({\bf u})$, which is to say that ${\cal T}({\bf r})$ is the same labeled multigraph as ${\cal T}({\bf u})$,  if and only if 
${\bf r}\in$ Seq$({\bf u})$. 

The multigraph ${\cal T}({\bf u})$ is simple, i.e., is a graph, if and only if {\bf u} is injective.\footnote{We call {\bf s} {\em injective} iff $s_i=s_j \Leftrightarrow i=j$ for $s_i$ and $s_j$  terms in {\bf s}; 
i.e., iff the function ${\bf s}:j\mapsto s_j \in {\rm Sym}(n)$ is injective.}  Where we omit the prefix ``multi'' from ``multentity'', we are tacitly indicating that the entity is simple; but our writing that 
$X$ is a multithing does not prohibit $X$ from its being a simple thing. (E.g., a multiedge can be of multiplicity $1$.)

For ${\cal T}({\bf u})$ a simple tree, this graph has been used in \cite{Eden} to specify the set of all ${\bf r}\in$ Seq$({\bf u})$ for which $\bigcirc{\bf r}=\bigcirc{\bf u}$, thus inducing a natural 
partition of Seq$({\bf u})$. Also, both \cite{Denes} and \cite{Silberger} show that, if the multigraph ${\cal T}({\bf u})$ is simple, then every element in Prod$({\bf u})$ is a cyclic permutation of the set 
$n$ if and only if ${\cal T}({\bf u})$ is a tree. See also \cite{Polya}.

\begin{defn}\label{DefPermComplete} A sequence {\bf s} in Sym$(n)$ is {\em permutationally complete} iff Prod$({\bf s})\in \{{\rm Alt}(n),{\rm Blt}(n)\}$ where ${\rm Blt}(n) := 
{\rm Sym}(n)\setminus{\rm Alt}(n)$; that is to say, ${\bf s}$ is permutationally complete iff ${\rm Prod}({\bf s})\in{\rm Sym}(n)/{\rm Alt}(n)$. `Permutationally complete' is abbreviated 
perm-complete.\end{defn}

A sequence ${\bf s}$ in ${\rm Sym}(n)$ is perm-complete if and only if ${\rm Prod}({\bf s})$ is of largest possible size, $|{\rm Prod}({\bf s})| = n!/2$.

In \S2 we elaborate criteria that imply the perm-completeness of a sequence {\bf u} in $1^{n-2}2^1$, and we provide other criteria which entail that such a {\bf u} cannot be perm-complete. 

If the product function $\bigcirc$ maps Seq$({\bf s})$ onto an element in the family ${\rm Sym}(n)/{\rm Alt}(n)$, and if {\bf r} is a sequence produced by inserting into {\bf s} an additional term 
$f\in{\rm Sym}(n)$, then plainly $\bigcirc$ maps ${\rm Seq}({\bf r})$ onto an element in ${\rm Sym}(n)/{\rm Alt}(n)$; {\it viz} Theorem \ref{SupSeq}. So we can confine our attention in \S2 to 
those transpositional {\bf u} which are injective, and whose transpositional multigraphs are consequently simple; i.e., they are ``graphs''.   

These graphs facilitate the identification of infinite classes of {\bf u} which are perm-complete and also of infinite classes of {\bf u} which fail to be perm-complete. For instance, if ${\cal T}({\bf u})$ 
is the complete graph ${\cal K}_n$ then {\bf u} is perm-complete, but if ${\cal T}({\bf u})$ is a tree with $n\ge3$ then {\bf u} is not perm-complete. Therefore, every injective perm-complete 
transpositional sequence {\bf u} has a minimal perm-complete subsequence. 

In \S2 we will specify, for each $n\ge2$, a family of minimal perm-complete injective sequences in $1^{n-2}2^1$.

\begin{defn}\label{DefCI} We call a permutational sequence {\bf s} {\em conjugacy invariant}, aka CI, iff every element in Prod$({\bf s})$ is conjugate to $\bigcirc{\bf s}$.\end{defn}

We lose no generality if we ignore the fact that ${\cal T}({\bf u})$ is labeled. Indeed, we call an unlabeled multigraph ${\cal G}$ perm-complete if ${\cal G}$ is isomorphic to ${\cal T}({\bf u})$ 
for some perm-complete ${\bf u}$. Likewise, ${\cal G}$ is CI if ${\bf u}$ is CI. 

\section{Permutational completeness}

\begin{thm} \label{SupSeq}  Every supersequence in {\rm Sym}$(n)$ of a perm-complete sequence {\bf s} in {\rm Sym}$(n)$ is perm-complete.    \end{thm} 

\begin{proof} Without loss of generality, let ${\rm Prod}({\bf s})={\rm Alt}(n)$. Pick $g\in{\rm Sym}(n)$. The mapping, ${\rm Alt}(n)\rightarrow{\rm Sym}(n)$ defined by $f\mapsto f\circ g$, takes 
${\rm Prod}({\bf s})$ either into ${\rm Alt}(n)$ or into ${\rm Blt}(n)$, and it is bijective. Since $|{\rm Alt}(n)|=n!/2=|{\rm Blt}(n)|$, we conclude that ${\rm Prod}({\bf w})\in{\rm Sym}(n)/{\rm Alt}(n)$ 
where ${\bf w} := \langle{\bf s},g\rangle=\langle s_0,s_1,\ldots, s_{k-1},g\rangle$.     \end{proof}

Call a family ${\cal A}$ {\em connected} if $\bigcup{\cal A}$ cannot be expressed as the disjoint union, $\bigcup{\cal A} = \bigcup{\cal B}\,\dot{\cup}\bigcup{\cal C}$, of nonempty subfamilies ${\cal B}$ 
and ${\cal C}$  of ${\cal A}$. Observe that if {\bf s} is perm-complete then ${\rm Supp}({\bf s})$ is connected.\vspace{.5em}

In \S2. we restrict our concern to those ${\bf u} := \langle u_0,u_1,\ldots, u_{k-1}\rangle$ in $1^{n-2}2^1$ for which $\bigcup{\rm Supp}({\bf u}) = n$, and for which the family ${\rm Supp}({\bf u})$ is connected. 
It is easy to see that if $g\in{\rm Prod}({\bf u})$ then $g^-\in{\rm Prod}({\bf u})$ too.\footnote{We write $g^-$ to designate the inverse of $g$, where other people may prefer instead to write $g^{-1}$.}

\subsection{Criteria ensuring that {\bf u} is not perm-complete} 

\begin{thm}\label{Deg1} Let ${\cal G}$ be a connected graph with vertex set $n\ge3$, and which has a vertex $a$ if degree $1$. Then ${\cal G}$ is not perm-complete. Consequently, no tree having three 
or more vertices is perm-complete.     \end{thm}

\begin{proof} Pretend that ${\cal G}$ is perm-complete, and let ${\bf u}$ be a transpositional sequence for which ${\cal G = T}({\bf u})$. Then there exists ${\bf r}\in{\rm Seq}({\bf u})$ with 
$a = a\bigcirc{\bf r}$. There are subsequences ${\bf f}$ and ${\bf g}$ of ${\bf r}$ such that ${\bf r} = \langle{\bf f},(a\,b), {\bf g}\rangle$ for some $b\in n\setminus\{a\}$. Since by hypothesis 
deg$_{\cal g}(a)=1$, the transposition $(a\,b)$ is the only term in the injective sequence ${\bf r}$ with $a$ in its support, we get that 
$a = a\bigcirc{\bf r} = a[\bigcirc{\bf f}\circ(a\,b)\circ\bigcirc{\bf g}] = a[(a\,b)\circ\bigcirc{\bf g}] = b\bigcirc{\bf g} \not= a$. \end{proof} 

A transpositional sequence {\bf u} in Sym$(n)$ with $3\le|{\bf u}|<n$ fails to be perm-complete, since $|$Prod$({\bf u})|\le |$Seq$({\bf u})|\le |{\bf u}|!<n!/2$. Thus Theorem \ref{Deg1} implies that 
there exist non-perm-complete injective {\bf u} of length \[{{n-1}\choose{2}}+1.\]

Although $|{\rm Seq}({\bf r})|=|{\bf r}|!$ when {\bf r} is an injective sequence in ${\rm Sym}(n)$, it is rare that $|{\rm Prod}({\bf r})|=|{\bf r}|!$ 

\begin{thm}\label{Deg2} Let ${\cal G}$ be a connected graph\footnote{For the notion of a connected graph, one may consult\cite{Harary} or almost any other textbook on graph theory.} on the vertex set $n\ge4$, and let ${\cal G}$ have adjacent vertices $x$ and $y$ each of which is of degree $2$. Then ${\cal G}$ is 
not perm-complete.     \end{thm}

\begin{proof} Pretend that ${\cal G}$ is perm-complete, and assume that ${\bf u}$ is an injective sequence in $1^{n-2}2^1$ with ${\cal T}({\bf u})$ isomorphic to ${\cal G}$. There are\footnote{not necessarily distinct} elements $a$ and $b$ in $n\setminus\{x,y\}$ such that $(a\,x),\,(x\,y)$, and $(y\,b)$ are edges 
of ${\cal G}$. So we let ${\bf r}\in{\rm Seq}({\bf u})$ satisfy  both $x = x\bigcirc{\bf r}$ and $y = y\bigcirc{\bf r}$. Let ${\bf r'}$ be the sequence of length $|{\bf u}|-3$ obtained by removing the terms $(a\,x), (x\,y)$ and $(y\,b)$ from ${\bf r}$. Let   
${\bf r' = fghk}$ be the factorization of ${\bf r'}$ into the four\footnote{some of which may be empty} consecutive segments engendered by the removal from ${\bf r}$ of those three terms. Of 
course $\{x,y\}\cap\big({\rm supp}(\bigcirc{\bf f})\cup{\rm supp}(\bigcirc{\bf g})\cup{\rm supp}(\bigcirc{\bf h})\cup{\rm supp}(\bigcirc{\bf k})\big) = \emptyset$. There are essentially three cases. \vspace{.5em}

\underline{Case}: \ $(a\,x) <_{\bf r} (x\,y) <_{\bf r} (y\,b)$. So ${\bf r} = \langle{\bf f},(a\,x),{\bf g},(x\,y),{\bf h},(y\,b),{\bf k}\rangle$. Then 
$x\bigcirc{\bf r} = x[\bigcirc{\bf f}\circ(a\,x)\circ\bigcirc{\bf g}\circ(x\,y)\circ\bigcirc{\bf h}\circ(y\,b)\circ\bigcirc{\bf k}] = 
x[(a\,x)\circ\bigcirc{\bf g}\circ(x\,y)\circ\bigcirc{\bf h}\circ(y\,b)\circ\bigcirc{\bf k}]  =  
a[\bigcirc{\bf g}\circ(x\,y)\circ\bigcirc{\bf h}\circ(y\,b)\circ\bigcirc{\bf k}] = a[\bigcirc{\bf h}\circ(y\,b)\circ\bigcirc{\bf k}]$.

{\it Subcase}: \ $b=a\bigcirc{\bf h}$. Then $x\bigcirc{\bf r} = y\bigcirc{\bf k} = y \not= x$ since $y\not\in{\rm supp}(\bigcirc{\bf k})$.

{\it Subcase}: \ $b\not= c := a\bigcirc{\bf h}$. Then $x\bigcirc{\bf r} = c\bigcirc{\bf hk} \not= x$.\vspace{.5em}

\underline{Case}: \ $(a\,x) <_{\bf r} (y\,b) <_{\bf r} (x\,y)$. Here, ${\bf r} = \langle{\bf f},(a\,x),{\bf g},(y\,b),{\bf h},(x\,y),{\bf k}\rangle$. Now  $y\bigcirc{\bf r} = b[\bigcirc{\bf h}\circ(x\,y)\circ\bigcirc{\bf k}] 
 = b\bigcirc{\bf hk} \not= y$.\vspace{.5em}

\underline{Case}: \ $(y\,b) <_{\bf r} (a\,x) <_{\bf r} (x\,y)$. So  ${\bf r = \langle f},(y\,b),{\bf g},(a\,x),{\bf h},(x\,y),{\bf k}\rangle$, and so $x\bigcirc{\bf r} =  a\bigcirc{\bf hk} \not= x$.\vspace{.5em}  

\noindent In each of these three cases we see that either $x\bigcirc{\bf r} \not= x$ or $y\bigcirc{\bf r} \not= y$, contrary to our requirement on ${\bf r}$. \end{proof}

\noindent{\bf Remarks.} Surely both of the complete graphs ${\cal K}_2$ and ${\cal K}_3$ are perm-complete. In fact, the triangle ${\cal K}_3$ is minimally so, in the sense that 
the removal of one edge produces a graph which is not perm-complete.\vspace{1em} 
 
Next, we prepare the way for two more-general theorems, each of which provides sufficient conditions for non-perm-completeness. \vspace{.5em}

Since, in the present context, the transformational multigraph of each minimal perm-complete sequence is simple, we let ${\cal G}$ be a simple connected graph whose vertex set is $n$, 

Fix a sequence ${\bf u} := \langle u_0,u_1,\ldots, u_k\rangle$ for which ${\cal G} = {\cal T}({\bf u})$, where $u_i := (x_i\,y_i)$ for each $i\le k$. 

$\overrightarrow{\cal G}$ denotes the digraph obtained by replacing each edge $(x\, y)$ of ${\cal G}$ with the two {\em arcs}  $x\rightarrow y$ and $x\leftarrow y$.

For $x\in n$, the {\bf u}-{\em path from $x$} is the subdigraph, $\overrightarrow{{\bf u}_x }:= \quad x\rightarrow z_1\rightarrow z_2\rightarrow\cdots\rightarrow z_m\rightarrow y$, of $\overrightarrow{\cal G}$, 
where the vertices of this path are chosen (and given new names) in the following fashion: 

Let $j(1)$ be the least integer $\ell$ such that $x\in{\rm supp}(u_\ell)$. So $u_{j(1)} = (x\, z)$ for some $z\in n$. Supposing the integers $j(1)<j(2)<\cdots<j(i)$ 
to have been chosen with $(z_{d-1}\, z_d) = u_{j(d)}$ for each $d\in\{2,3,\ldots, i\}$, let $j(i+1)$ be the smallest integer $v>j(i)$ with $z_i\in$ supp$(u_v)$ if any such $v$ exists, 
and in this event, define  $(z_i\, z_{i+1}) := u_{j(i+1)}$; but if there is no such $v$ then define $\langle z_m,y\rangle := \langle z_{i-1},z_i\rangle$. 

Let ${\bf u}_x$ be the subsequence of ${\bf u}$ whose terms contribute the respective arcs that comprise  $\overrightarrow{{\bf u}_x}$. 

\begin{lem}\label{PartitionArcs} \ $\{\overrightarrow{{\bf u}_x}: x\in n\}$ is a partition of the set of arcs comprising the digraph $\overrightarrow{\cal T}({\bf u})$.  \end{lem} 

\begin{proof} Since $n = \bigcup{\rm Supp}({\bf u})$, we have that ${\bf u}_x\not=\emptyset$ for every $x\in n$. Let $a\rightarrow b$ be an arc in $\overrightarrow{\cal T}({\bf u})$. Then 
$(a\, b)$ is a term $u_i$ in the sequence {\bf u}. If $i=0$ then let $g$ be the identity permutation $\iota\restrict n$; but, if $i>0$, let $g := u_0\circ u_1\circ\cdots\circ u_{i-1}$. Let $v:=ag^-$.Then  
$vg=a$, and so $a\rightarrow b$ is an arc of ${\bf u}_v$. Furthermore, if $v\not= q\in n$, then $qg\not=a$ and thus $a\rightarrow b$ is not an arc of ${\bf u}_q$.    \end{proof} 

\begin{defn}\label{Minimal} A perm-complete sequence ${\bf s}$ in ${\rm Sym}(n)$ is {\em minimally perm-complete} iff the removal of any term of ${\bf s}$ results in a sequence which is not perm-complete. \end{defn}

\begin{defn}\label{Cutset} A set $C$ of edges of a connected graph ${\cal G}$ is a {\em cut set} of ${\cal G}$ iff the removal of $C$ from the edge set of ${\cal G}$ results in a graph that is the union ${\cal G}_0\,\dot\cup\, {\cal G}_1$ of two disjoint  subgraphs of ${\cal G}$, with each edge in $C$ having one vertex in ${\cal G}_0$ and the other in ${\cal G}_1$.\end{defn}

The next two theorems facilitate the identification of non-perm-complete transpositional sequences. 

\begin{thm}\label{NotPermComplete1} Let ${\cal G}:=\langle n;E\rangle$ be a simple connected graph whose vertex set is $n$ and whose edge set is $E$, and which has a cut set 
$C\subseteq E$. Let ${\cal G}_0 :=\langle V_0;E_0\rangle$ and ${\cal G}_1 := \langle V_1;E_1\rangle$ be the disjoint subgraphs of ${\cal G}$ gained by the removal of $C$ from $E$, 
where $V_i$ and $E_i$ are respectively the vertex sets and the edge sets of the two ${\cal G}_i$. Let ${\cal G}_0$ be a forest, let $2|C|<|V_0|$, and let $|V_1|\ge2$. Then ${\cal G}$ fails 
to be perm-complete.        \end{thm}

\begin{proof} Assume that ${\cal G}$ is perm-complete. Then ${\cal G} = {\cal T}({\bf s})$ for some injective sequence ${\bf s} := \langle s_0,s_1,\ldots, s_{k-1}\rangle$ of transpositions in 
${\rm Sym}(n)$, where of course $|E|=k$. If $|E|$ is even then ${\rm Prod}({\bf s})={\rm Alt}(n)$; so $\iota\restrict n\in{\rm Prod}({\bf s})$. But if $|E|$ is odd then 
$(u\, w)\in{\rm Prod}({\bf s})$ for some elements $u\not= w$  in $V_1$. In both cases there exists $f\in{\rm Prod}({\bf s})$ with $xf=x$ for every $x\in V_0$. The perm-completeness of 
${\bf s}$ implies that $f=\bigcirc{\bf t}$ for some  ${\bf t}\in{\rm Seq}({\bf s})$.  

Let $x$ be an arbirary element in $V_0$.

Since $x = x\bigcirc{\bf t}$, the ${\bf t}$-path $\overrightarrow{{\bf t}_x}$ induces a directed circuit $\overrightarrow{{\bf t}_x^\star}$ that starts and ends at $x$. So, since the subgraph ${\cal G}_0$ is a forest, and since the sequence ${\bf t}$ is injective, we can show that the path $\overrightarrow{{\bf t}_x}$ uses up two edges $e_x \not= l_x$ in $C$ that contribute, to $\overrightarrow{{\bf t}_x^\star}$, an arc $\overrightarrow{e_x}$ from $V_0$ to $V_1$ and another arc $\overrightarrow{l_x}$ back from $V_1$ to $V_0$.

Pretend that the arc $\overrightarrow{e_x}$ occurs not only in the path $\overrightarrow{{\bf t}_x}$, but also in the path $\overrightarrow{{\bf t}_{x'}}$ for some $x'\in V_0\setminus\{x\}$. Then Lemma \ref{PartitionArcs} implies 
that the set of arcs comprising $\overrightarrow{{\bf t}_{x'}}$ is the same collection of arcs that comprise $\overrightarrow{{\bf t}_x}$. Viewed as a subsequence of ${\bf t}$, the word ${\bf t}_{x'}$ is a cyclic conjugate of the word ${\bf t}_{x}$.   

Without loss of generality, take it that $h' \le_{\bf t} h$, where the transpositions $h'$ and $h$ are the first terms in ${\bf t}$, under the ordering $\le_{\bf t}$, to have $x'$ and $x$ in their respective supports.

Assume that $h'=h=(x\,x')$. Then $x\in {\rm supp}(h'_1)$ where $h'_1$ is the term immediately following $h'$ in the subsequence ${\bf t}_{x'}$ of ${\bf t}$. Similarly, $x'\in{\rm supp}(h_1)$, where $h_1$ is the immediate successor of $h$ in the subsequence ${\bf t}_x$. But obviously then $h'_1=h_1=(x\,x')=h$ in violation of the injectivity of the sequence ${\bf t}$. It follows that $h'<_{\bf t} h$. So there is a prefix  
${\bf p}:= \langle h', h'_1, h'_2, \ldots, h'_\bullet, h\rangle$ of the word ${\bf t}_{x'}$, which induces  a digraph $\overrightarrow{{\bf p}^\star}$ whose vertices are the integer endpoints of the arcs in $\overrightarrow{\bf p}$, 
and which extends from $x'\in{\rm supp}(h')$ to $x\in{\rm supp}(h)$. 

Of course ${\bf p}$ is a subsequence of ${\bf t}$. Notice that ${\rm supp}(h'_\bullet)\cap{\rm supp}(h)= \{x\}$, and that  $t_0 \le_{\bf t} h'\le_{\bf t}h'_\bullet <_{\bf t} h$, where $t_0$ is the first term in the sequence ${\bf t}$. That $x\in{\rm supp}(h'_\bullet)$ violates a manufacturing criterion for the sequence ${\bf t}_x$; to wit: Under the ordering $<_{\bf t}$, the first term in ${\bf t}_x$ was specified to be the first term in the sequence ${\bf t}$, having $x$ in its support. That first term of ${\bf t}_x$ is $h>_{\bf t} h'_\bullet$. So $f\not\in{\rm Prod}({\bf s})$. Having verified that each vertex in $V_0$ uses up (at least) two edges in $C$ if indeed $f=\bigcirc{\bf t}$, we infer that 
$|V_0|\le 2|C|$ if ${\bf s}$ is perm-complete. So, since $2|C|<|V_0|$ by hypothesis, we conclude that ${\bf s}$ is not perm-complete.    \end{proof} 

A modification of the proof of Theorem \ref{NotPermComplete1} will establish

\begin{cor}\label{WeakerNPC1} Let the hypothesis $2|C|<V_0$ in Theorem \ref{NotPermComplete1} be replaced by the hypothesis $|C|\le V_0$, but let the other hypotheses of the theorem hold. Then  ${\cal G}$ fails to be  
perm-complete.\end{cor} 

\begin{thm}\label{NotPermComplete2} Let ${\cal G}_0$ and ${\cal G}_1$ be connected graphs on the disjoint vertex sets $V_0$ and $V_1$, with $V_0\cup V_1=n$ and $\min\{|V_0|,|V_1|\}\ge2$. 
Let $C$ be a nonempty set of edges, each of which has one of its vertices in $V_0$ and the other in $V_1$. Let ${\cal G}:=\langle n;E\rangle={\cal G}_0\,{\cup}\,C\,{\cup}\,{\cal G}_1$.  
Let $|C|< \min\{|V_0|,|V_1|\}$. Then ${\cal G}$ is not perm-complete.      \end{thm} 

\begin{proof} Let $c := |C|<\min\{m,p\}$ where $V_0=\{x_0,x_1,\ldots, x_{m-1}\}$ and where $V_1=\{y_0,y_1,\ldots, y_{p-1}\}$. Assume that ${\cal G}$ is perm-complete. Then 
${\cal G = T}({\bf s})$ for some sequence {\bf s} of transpositions in Sym$(n)$.\vspace{.5em} 

\underline{Case}: $|E|$ is odd. Let $f:=(x_0\, y_0\, x_1\, y_1\,\ldots\, x_c\, y_c) = f\in{\rm Blt}(n)$. Choose ${\bf r}\in$ Seq$({\bf s})$ such that $f=\bigcirc{\bf r}$. By Lemma 
\ref{PartitionArcs}, each of the $2c+2$ distinct paths ${\bf r}_z$ in $\overrightarrow{\cal G}$, one for each $z\in\{x_0,y_0,\ldots, x_c,y_c\}$, contains an arc in $\overrightarrow{\cal C}$ that is 
contained in no ${\bf r}_{z'}$ with $z'\in\{x_0,y_0,\ldots, x_c,y_c\}\setminus\{z\}$. But $\overrightarrow{\cal C}$ has only $2c$ arcs in all. Hence,  $f\notin$ Prod$({\bf s})$. Thus we see that 
${\cal G}$ fails to be perm-complete in the case that $|E|$ is odd. \vspace{.5em}

\underline{Case}: $|E|$ is even. \vspace{.3em}

{\it Subcase}: $c$ is odd. Let $g := (x_0\, y_0)(x_1\, y_1)\dots(x_c\, y_c)\in{\rm Alt}(n)$. Choose ${\bf w}\in{\rm Seq}({\bf s})$ for which $g=\bigcirc{\bf w}$. 
As in the odd $|E|$ case, each of the $2c+2$ paths ${\bf w}_z$ in $\overrightarrow{\cal G}$ for the $z\in\{x_0,y_0,x_1,y_1,\ldots, x_c,y_c\}$ uses an arc in $\overrightarrow{\cal C}$ that is 
contained in no path ${\bf w}_{z'}$ with $z'\in\{x_0,y_0,x_1,y_1,\ldots, x_c,y_c\}\setminus\{z\}$ -- an impossibility since $\overrightarrow{\cal C}$ has only $2c$ arcs. So $g\notin$ Prod$({\bf s})$. 
We infer that here too ${\cal G}$ is not perm-complete.\vspace{.3em}

{\it Subcase}: $c$ is even. We amalgamate two $2$-cycles of $g$ to create a $4$-cycle, thus producing the even permutation $h := (x_0\, y_0\, x_1\, y_1)(x_2\, y_2)\ldots(x_c\, y_c)$. 
Having assumed $h\in$ Prod$({\bf s})$, we can choose ${\bf u}\in$ Seq$({\bf s})$ for which $h=\bigcirc{\bf u}$. Once again we have that the set of $2c+2$ paths ${\bf u}_z$ is obliged to 
use $2c+2$ arcs in $\overrightarrow{\cal C}$, but cannot do so since $\overrightarrow{\cal C}$ has only $2c$ arcs. Again we get that ${\cal G}$ is not perm-complete.  \end{proof}

\subsection{Criteria ensuring permutational completeness}

When ${\cal G} := \langle n;E\rangle$ is a graph with vertex set $n$ and edge set $E$, and when $W\subseteq n$, then $\langle W\rangle$ denotes the subgraph $\langle W;D\rangle$ of ${\cal G}$ 
whose vertex set is $W$, and whose edge set $D$ consists of every edge $(x\, y)\in E$ for which $\{x,y\}\subseteq W$. This subgraph $\langle W\rangle$ of ${\cal G}$ is said to be {\em induced} by 
$W$ in ${\cal G}$.

We say that a subgraph  ${\cal S}$ of a graph ${\cal H}$ {\em spans} ${\cal H}$ iff the vertex set of ${\cal S}$ is that of ${\cal H}$. If a subgraph ${\cal S}$ of ${\cal H}$ spans ${\cal H}$, and if 
no two distinct edges of ${\cal S}$ share a vertex, then we call ${\cal S}$ a {\em perfect matching} for ${\cal H}$.

\begin{thm}\label{Complicated} For {\bf t} an injective perm-complete transpositional sequence in {\rm Sym}$(n)$, let ${\cal G} := \langle n;E\rangle = {\cal T}({\bf t})$. Let $\emptyset \not= 
W\subseteq n$, and let $x\notin n$ be a new vertex. Let ${\cal H} :=\langle V_0;E_0\rangle$ be the simple supergraph of ${\cal G}$ for which $V_0 := n\cup\{x\}$ is the vertex set of ${\cal H}$, and 
where $E_0 := E\cup\{(x\,w): w\in W\}$ is the edge set of ${\cal H}$. Let {\bf s} be an injective transpositional sequence in {\rm Sym}$(V_0)$ such that $(p\, q)$ is a term of {\bf s} if and only if 
$(p\, q)\in E_0$. Let the integer $|E_0|$ be even(odd). Given a permutation $f\in{\rm Sym}(V_0)$ that is, correspondively, even(odd): 

{\bf 2.8.1} \ If $xf\in W$ then $f\in$ {\rm Prod}$({\bf s})$. 

{\bf 2.8.2} \ If $w_0f=w_1\not=w_0$ for some $\{w_0,w_1\}\subseteq W$, then $f\in$ {\rm Prod}$({\bf s})$. 

{\bf 2.8.3} \ If $\langle W\rangle$ contains a perfect matching, and if $xf=x$ as well, then $f\in$ {\rm Prod}$({\bf s})$.         \end{thm}  

\begin{proof}  We establish the theorem for the case where $|E_0|$ is even, and omit the (identical) proof for the case where $|E_0|$ is odd. So now let $|E_0|$ be even. Since ${\cal G}$ 
is perm-complete, we have that $|{\rm Prod}({\bf t})|=n!/2$.  

Let $W := \{w_0,w_1,\ldots, w_{k-1}\}\subseteq n$ with $|W|=k$.  We write $f^+ := f\cup\{\langle x,x\rangle\}\in$ Sym$(V_0)$; \ i.e., $f^+$ is just $f$ 
augmented by the 1-cycle $(x)$.\vspace{.5em}

 To prove 2.8.1, let  \[Q := \{h^+\circ(w_0\, x)\circ(w_1\, x)\circ\cdots\circ(w_{k-1}\, x): h\in\,\mbox{Prod}({\bf t})\}.\] Define $\varphi:$ Prod$({\bf t})\rightarrow Q$ by 
$\varphi(h) := h^+\circ(w_0\, x)\circ(w_1\, x)\circ\cdots\circ(w_{k-1}\, x)$. Plainly $\varphi$ is a bijection from Prod$({\bf t})$ onto $Q$. It follows that $|Q|=n!/2$. Now let $M := \{g: xg=w_0$ and 
$g\in$ Alt$(V_0)\}$. Observe that $Q\subseteq M$. 

Given $g\in M$, we have $\{\langle x,w_0\rangle,\langle z_g,x\rangle\}\subseteq g$ for some $z_g\in n$. Let $g^* := (g\setminus\{\langle x,w_0\rangle,\langle z_g,x\rangle\})\cup\{\langle z_g,w_0\rangle\}$. 
The function $^*:g\mapsto g^*$ obviously maps $M$ bijectively onto Blt$(n)$. Hence $|M|=n!/2$. Therefore $Q=M$. But $Q\subseteq$ Prod$({\bf s})$. The assertion 2.8.1 follows.\vspace{.5em}   

To prove 2.8.2, let $P := \{(w_0\, x)\circ h^+\circ(w_1\, x)\circ(w_2\, x)\circ\cdots\circ(w_{k-1}\, x): h\in{\rm Prod}({\bf t})\}$. Define the function $\psi:{\rm Prod}({\bf t})\rightarrow P$ by $\psi(h) :=         
(w_0\, x)\circ h^+\circ(w_1\, x)\circ(w_2\, x)\circ\cdots\circ(w_{k-1}\, x)$. Notice that $\psi$ is a bijection from ${\rm Prod}({\bf t})$ onto $P$. So $|P|=n!/2$. Let $L:=\{g: w_0g=w_1$ and $g\in{\rm Alt}(V_0)\}$. 
Then $P\subseteq L$.  

For $g\in L$, let $y_g:=w_1g$. Let $g^\dag := (g\setminus\{\langle w_0,w_1\rangle,\langle w_1,y_g\rangle\})\cup\{\langle w_0,y_g\rangle\}$. The function $^\dag:g\mapsto g^\dag$ obviously maps 
$L$ bijectively onto Blt$(V_0\setminus\{w_1\})$. However, $|V_0\setminus\{w_1\}| = n$. So $|L|=n!/2$. Thus $P=L$. But $P\subseteq$ Prod$({\bf s})$. The assertion 2.8.2 follows. \vspace{.5em}

To prove 2.8.3, take $|W|=k=2m\ge2$ to be even, and let $A:=\{(x_0\,y_0),(x_1\,y_1),\ldots,(x_{m-1}\,y_{m-1})\}$ be a perfect matching of $\langle W\rangle$. Since ${\cal H}$ has an even number 
of edges, ${\cal G}$ also has an even number of edges. Thus Prod$({\bf t}) =$ Alt$(n)$. So it suffices to show for each $h\in{\rm  Alt}(n)$ that $h^+=\bigcirc{\bf s}\in{\rm Prod}(V_0)$ for some sequence 
{\sf s} such that ${\cal H=T}({\bf s})$.  

Let $h\in{\rm Alt}(n)$. Choose ${\bf r}\in{\rm Seq}({\bf t})$ such that $h=\bigcirc{\bf r}$. We expand the length-$|{\bf t}|$ sequence {\bf r} to a sequence {\bf s} of $|{\bf t}|+2m$ distinct transpositions in Sym$(V_0)$, by 
replacing each of the $m$ special terms, $(x_i\, y_i)$, in {\bf r} with the corresponding three-term sequence $\langle(x\, x_i),(x_i\, y_i),(y_i\, x)\rangle$. Plainly $h^+=\bigcirc{\bf s}$. Therefore $h^+\in{\rm Prod}({\bf s})$. 
The assertion 2.8.3 follows.  \end{proof}

\begin{cor}\label{Rectangle} A rectangle with one of its two diagonals is a minimal perm-complete transpositional graph.      \end{cor}

\begin{proof} Let ${\bf t} := \langle(0\, 1),(0\, 2),(0\, 3),(1\, 3),(2\, 3)\rangle$. It is obvious from Theorems \ref{Deg1} and \ref{Deg2} that the removal of any of the five terms of {\bf t} results in a 
transpositional sequence in Sym$(4)$ which is not perm-complete. Therefore it suffices to show that {\bf t} itself is perm-complete. 

Since the triangle graph is perm-complete, Theorem \ref{Complicated} implies that Prod$({\bf t})$ contains every $h\in$ Blt$(4)$ except possibly for the missing diagonal, $(1\, 2)$. But 
$(1\, 2) = (0\, 1)\circ(2\, 3)\circ(0\, 2)\circ(1\, 3)\circ(0\, 3)\in$ Prod$({\bf t})$.      \end{proof}

By a {\em bike} on $n+2$ vertices we mean any graph isomorphic to the labeled graph ${\cal B}_n$, whose edge set has these 2n+1 edges: the ``axle'' $(0\, 1)$ and the $2n$ ``spokes'' 
$(0\, i)$ and $(1\, i)$ for the $i\in\{2,3,\ldots, n, n+1\}$. 

We already observed that the tree with one edge, ${\cal B}_0={\cal K}_2$, and the triangle, ${\cal B}_1={\cal K}_3$, are minimal perm-complete. By Corollary \ref{Rectangle} we have that 
the proper subgraph ${\cal B}_2$ of ${\cal K}_4$ is minimal perm-complete.\vspace{.5em}

As usual, $\omega := \{0,1,2,\ldots\}$. Let $\langle x_1,x_2,\ldots\rangle$ be an injective sequence in $\omega\setminus 2 := \{2,3,4,\ldots\}$. We recursively define an infinite sequence 
$\langle{\bf c}_{(2t)}\rangle_{t=1}^\infty$ of finite sequences of transpositions in Sym$(\omega)$ thus:

${\bf c}_{(2)} := \langle(1\, x_2),(0\, x_2)\rangle$      

${\bf c}_{(2t+2)} := \langle{\bf c}_{(2t)},(0\, x_{2t+1}),(1\, x_{2t+1}),(1\, x_{2t+2}),(0\, x_{2t+2})\rangle$ 

\begin{lem}\label{Bike1} Let ${\bf r}_{(2t)} := \langle(0\, x_1),{\bf c}_{(2t)},(1\, x_1)\rangle$ for $t>0$ an integer. Then $\bigcirc{\bf r}_{(2t)} = (0\, 1)(x_1\, x_2\,\cdots x_{2t})$.    \end{lem}

\begin{proof} Since $\bigcirc{\bf r}_{(2)}=(0\, x_1)\circ(1\, x_2)\circ(0\, x_2)\circ(1\, x_1)= (0\, 1)(x_1\, x_2)$, the basis holds for an induction on $t$.

Now pick $t\ge1$, and suppose that $\bigcirc{\bf r}_{(2t)}=(0\, 1)(x_1\, x_2\,\ldots\,x_{2t})$. Then 
\[ \bigcirc{\bf r}_{(2t+2)} = \bigcirc{\bf r}_{(2t)}\circ(1\, x_1)\circ(0\, x_{2t+1})\circ(1\, x_{2t+1})\circ(1\, x_{2t+2})\circ(0\, x_{2t+2})\circ(1\, x_1)= \]     
\[(0\, 1)(x_1\, x_2\, \ldots\, x_{2t})\circ(1\, x_1)\circ(0\, x_{2t+1})\circ(1\, x_{2t+1})\circ(1\, x_{2t+2})\circ(0\, x_{2t+2})\circ(1\, x_1)  = (0\, 1)(x_1\, x_2\,\ldots\,x_{2t+1}\, x_{2t+2}).  \] 
So $\bigcirc{\bf r}_{(2t+2)} = (0\, 1)(x_1\, x_2\,\ldots\,x_{2t+2})$.      \end{proof}

\begin{thm}\label{Bike2} ${\cal B}_n$ is a minimal perm-complete graph for every nonnegative integer $n$.    \end{thm}

\begin{proof} Recall that the theorem holds for $0\le n\le2$. So we will establish it for $n\ge3$. We show that the removal of an edge from ${\cal B}_n$ results in a subgraph which fails to 
be perm-complete. So, if ${\cal B}_n$ is perm-complete then it is minimal as such. 

The removal of a spoke from ${\cal B}_n$ results in a subgraph that has a vertex of degree $1$. By Corollary \ref{Deg1}, such a subgraph is not perm-complete. So consider the subgraph 
${\cal G}_n := {\cal B}_n-(0\, 1)$ obtained by removing the axle from ${\cal B}_n$. Now ${\cal G}_n = {\cal G}_{n,0}\dot{\cup}{\cal E}\dot{\cup}{\cal G}_{n,1}$ is a disjoint union, 
where ${\cal G}_{n,1}$ is the one-edge subgraph $(0\, 2)$, where ${\cal G}_{n,0}$ is the tree on the $n$ vertices -- $1,3,4,\ldots n,n+1$ -- and whose edge set is $\{(1\, j): 3\le j\le n+1$, 
and where ${\cal E}$ is the subgraph whose vertex set is all of $n+2$ and whose edge set is $C := \{(1\, 2)\}\cup\{(0\, j):3\le j\le n+1\}$. But $C$ is the cut set connecting ${\cal G}_{n,0}$ 
to ${\cal G}_{n,1}$ to form ${\cal G}_n$. So Corollary \ref{WeakerNPC1} implies ${\cal G}_n $ is not perm-complete.  

It remains only to show that ${\cal B}_n$ perm-complete. The basis of an induction is already established. So pick an integer $n\ge3$, and suppose for any nonnegative $i<n$ that any graph 
isomorphic to ${\cal B}_i$ is perm-complete. Let {\bf s} be a transpositional sequence in Sym$(n+2)$ such that ${\cal B}_n={\cal T}({\bf s})$.

Of course Prod$({\bf s})\subseteq$ Blt$(n+2)$. But we do need to show that Blt$(n+2)\subseteq$ Prod$({\bf s})$.\vspace{.3em}

\underline{Claim}: For every even positive integer $2t\le n$, the set Prod$({\bf s})$ contains every $f\in$ Blt$(n+2)$ which has a cyclic component of length $2t$. \vspace{.3em}

To prove this Claim, pick $2t\in\{2,3,\ldots, n\}$. Let $\langle x_1,x_2,\ldots, x_{2t-1},x_{2t}\rangle$ be any injective sequence in the set $\{2,3,\ldots, n,n+1\}$, and let $X$ be the $(2t)$-membered 
set $\{x_1,x_2,\ldots, x_{2t}\}$. Pick a sequence {\bf v} of transpositions such that ${\cal B}_{n\setminus X} = {\cal T}({\bf v})$, where ${\cal B}_{n\setminus X}$ is the graph obtained by 
removing the $2t$ vertices in $X$ from ${\cal B}_n$. Since ${\cal B}_{n\setminus X}$ is isomorphic to ${\cal B}_{n-2t}$, we have by the inductive hypothesis that ${\cal B}_{n\setminus X}$ is 
perm-complete. It follows that Prod$({\bf v})=$ Blt$((n+2)\setminus X)$.

Let {\bf r} be the transpositional sequence ${\bf r}_{(2t)}$ of Lemma \ref{Bike1}. Define $Q:=\{\bigcirc{\bf r}\circ g: g\in$ Blt$(n\setminus X)\}$. Then, by Lemma \ref{Bike1} we get that 
$Q = \{(x_1\, x_2\, \ldots x_{2t})(0\,1)\circ g: g\in$ Blt$(n\setminus X)\}$. Furthermore, $Q\subseteq$ Blt$(n+2)$. For each $g\in$ Blt$(n\setminus X)$, the concatenation ${\bf rv}_g$ is 
an element in Seq$({\bf s})$, where $g=\bigcirc{\bf v}_g$ for some ${\bf v}_g\in$ Seq$({\bf v})$.  Therefore $Q\subseteq$ Prod$({\bf s})$. Thus, when both $f\in$ Blt$(n+2)$, and $f$ has 
an even length cycle whose support is a subset of $\{2,3,\ldots, n+1\}$, then $f\in$ Prod$({\bf s})$. 

For every $x\in \{2,3,\ldots, n+1\}$, the graph ${\cal T}({\bf a}_x) := {\cal B}_{n\setminus\{x\}}$ is perm-complete by the inductive hypothesis, and hence by Theorem \ref{Complicated}.1 we 
have that Prod$({\bf a}_x)$ contains every $f_x\in$ Blt$((n+2)\setminus\{x\})$ such that $xf_x=x$; those $f_x$ include every one with an even-length cyclic component in $(n+2)\setminus\{x\}$. 
The claim is established.\vspace{.3em}

The theorem follows from the Claim, since every $f\in$ Blt$(n+2)$ has at least one even-length cycle.     \end{proof}

We call a vertex $v$ of a graph ${\cal G}$ {\em central} iff $v$ is adjacent to every other vertex of ${\cal G}$.

\begin{cor}\label{Central} If a connected graph ${\cal G}$ has at least two central vertices then ${\cal G}$ is perm-complete. \end{cor}

\begin{proof} The corollary is immediate by Theorems \ref{Bike2} and \ref{SupSeq}. \end{proof}

\begin{cor}\label{Complete} Every finite complete graph is perm-complete.  \end{cor}

The following examples provide instances where  the converse of Corollary \ref{Central} fails. 

\begin{prop}\label{Counterexamples1} Each of the following five transpositional sequences is minimally perm-complete: 
\[{\bf a} := \langle(0\, 1),(0\, 2),(0\, 3),(0\, 4),(1\, 2),(2\, 3),(3\, 4)\rangle\qquad {\bf b} := \langle {\bf a},(2\, 5),(3\, 5)\rangle \] 
\[{\bf c}:=\langle(0\,1),(0\,2),(0\,3),(1\,2),(1\,4),(2\,3),(3\,4)\rangle\qquad{\bf d} := \langle{\bf b},(4\,6),(5\,6)\rangle\]   \end{prop}

\noindent{\bf Partial Proof.} \ We shall establish our claim about {\bf a}, and leave the other four sequences for our reader. 

It is easy to see by Corollaries \ref{Deg1} and \ref{Deg2} that the removal of an edge from the graph ${\cal T}({\bf a})$ produces a graph which is not perm-complete. So it remains only show that 
{\bf a} is perm-complete. 

${\cal B}_2$ is perm-complete. Referring to Theorem \ref{Complicated}, identify ${\cal G}$ to be the copy of ${\cal B}_2$ whose vertex set is $\{0,1,2,3\}$, whose $W$ is $\{0,3\}$, and whose $x$ 
is the vertex $4$. By Theorem \ref{Complicated} and symmetry considerations, it is easy to see that Prod$({\bf a})$ contains every element in Blt$(5)$ except maybe $(1\,4)$. But, since 
$(1\,4)=(0\,2)\circ(3\,4)\circ(0\,1)\circ(2\,3)\circ(0\,4)\circ(1\,2)\circ(0\,3)$, we have that $(1\,4)\in$ Prod$({\bf a})$.  So {\bf a} is minimally perm-complete.\vspace{.5em}

Lest it be surmised that every graph which is an amalgamation of triangles is perm-complete, we offer

\begin{prop}\label{Counterexample2} Let  ${\bf e}:=\langle(0\,1),(0\,3),(1\,2),(1\,3),(2\,3),(2\,4),(3\,4),(3\,7),(4\,5),(4\,7),(5\,6),(5\,7),(6\,7)\rangle$. The transpositional sequence {\bf e} is not 
perm-complete.\end{prop}

\begin{proof} ${\cal T}({\bf e})$ consists of two copies of ${\cal B}_2$ conjoined by a three-element cut set. So Theorem \ref{NotPermComplete2} implies that ${\cal T}({\bf e})$ is not 
perm-complete.  \end{proof}    

By an $n$-{\em wheel} we mean any graph isomorphic to ${\cal W}_n:=\langle n+1;E\rangle$, where $E$ contains the following $2n$ edges: $(0\,i)$ for every $i\in\{1,2,\ldots, n\}$ and $(i\,i+1)$ for 
every $i\in\{1,2,\ldots, n-1\}$ and finally also $(1\,n)$. 

By Corollary \ref{WeakerNPC1} and Theorem \ref{NotPermComplete2}, if ${\cal W}_n$ is perm-complete then ${\cal W}_n$  is minimally perm-complete. \vspace{.5em}

\noindent{\bf Conjecture.} \ ${\cal W}_n$ is perm-complete for every $n\ge3$.

\section{Conjugacy invariance}

The present section will lay the ground work for, and thereafter establish, the following characterization of the conjugacy invariant transpositional sequences having multigraphs on the vertex set $n$ that are connected. 

\begin{thm}\label{CCIT} Let  {\bf u} be a transpositional sequence in {\rm Sym}$(n)$ with $2\le n\in{\mathbb N}$ whose multigraph  ${\cal T}({\bf u})$ is connected on the vertex set $n$. 
If $n=2$ then {\bf u} is both perm-complete and {\rm CI}. If $n=3$  then {\bf u} is {\rm CI} if and only if either $|{\bf u}|$ is odd or ${\cal T}({\bf u})$ is a multitree with at least one simple multiedge. 

For $n\ge4$, the sequence {\bf u} is {\rm CI} if and only if ${\cal T}({\bf u})$ is a multitree in which no vertex is an endpoint of more than one non-simple multiedge, and in which each even-multiplicity 
multiedge is a multitwig whose non-leaf vertex has only one non-leaf neighbor.\end{thm}

\subsection{Constant-product sequences}

We say that a permutational sequence {\bf s} is {\em constant-product} iff $|$Prod$({\bf s})|=1$. The class of constant-product {\bf s} is antipodal to the class of perm-complete {\bf s}. 

It is clear that ${\bf s} := \langle s_0,s_1,\ldots,s_k\rangle$ is constant-product if $s_i\circ s_j=s_j\circ s_i$ whenever $0\le i<j\le k$. Moreover, $s_i\circ s_j=s_j\circ s_i$ if either 
supp$(s_i)\,\cap\,$supp$(s_j)=\emptyset$ or $s_i$ and $s_j$ are powers $s_i=f^p$ and $s_j=f^q$ of a common permutation $f$; that is to say, {\bf s} is constant-product if {\bf s} is boring. 

Do there exist non-boring constant-product permutational sequences? 

We paraphrase a theorem of Eden and Sch\"{u}tzenberger (Page 144 of \cite{Eden}), which remarks upon certain injective transpositional sequences {\bf u}, and which touches on this question.\vspace{.5em} 

For each $v\in n$, let ${\bf u}_{(v)}$ be the subsequence of {\bf u} of which $u_{(v),j}$ is a term if and only if $v\in$ supp$(u_{(v),j})$. \vspace{.5em}  

{\bf Eden-Sch\"{u}tzenberger Theorem.} \ {\sl When the transpositional multigraph ${\cal T}({\bf u})$ is a simple tree and also ${\bf s}\in$ {\rm Seq}$({\bf u})$, then $\bigcirc{\bf s}=\bigcirc{\bf u}$ if and 
only if ${\bf s}_{(v)}={\bf u}_{(v)}$ for every $v\in n$.  }\vspace{.5em}

The paucity of non-boring constant-product permutational sequences, raises our interest to its superclass ${\cal O}(n)$ of permutational sequences {\bf s} for which the order of the permutation  
$\bigcirc{\bf x}$ is constant over all ${\bf x}\in$ Seq$({\bf s})$. The class of conjugacy invariant sequences is a natural proper subclass of ${\cal O}(n)$.

\subsection{Preliminaries}   

We call a binary relation $a\subseteq X\times X$ {\em conjugate} to $b\subseteq X\times X$, and write $a\simeq b$, iff $b=\{\langle xf,yf\rangle: \langle x,y\rangle\in a\}$ for some permutation 
$f\in$ Sym$(X)$. Equivalently, $a\simeq b$ iff $g^-\circ a\circ g = b$ for some $g\in$ Sym$(X)$. Plainly $\simeq$ is an equivalence relation on the family ${\cal P}(X\times X) := \{r: r\subseteq X\times X\}$ of 
all binary relations on the set $X$. 

We define the {\em world} of $c\subseteq X\times X$ to be $\$(c):=$ Dom$(c)\,\cup\,$Rng$(c)$. It is commonplace that $a\simeq b$ if and only if $b=g^-\circ a\circ g$ for some $g\in$ Sym$(\$(a)\cup\$(b))$. 
Of course $b = g^-\circ a\circ g$ if and only if $g\circ b=a\circ g$.\vspace{.5em}  

In this paper we restrict our attention to those binary relations which are permutations on the set $n$. Whenever $\{a,b\}\subseteq$ Sym$(n)$, we have not only that $a\circ b\simeq b\circ a$ but also 
that $a\simeq a^-$. 

For $n>0$ an integer, $[n]$ denotes the set  $\{1,2,\ldots,n\}$. (But remember that $n$ denotes $\{0,1,\ldots,n-1\}$.)

Type$(a)\subseteq{\cal P}(X\times X)$ denotes the conjugacy class of the binary relation $a\subseteq X\times X$. When $a\in$ Sym$(n)$ then Type$(a)$ acquires a more informative moniker; 
namely, Type$(a) := 1^{e(1)}2^{e(2)}\cdots n^{e(n)}$, where for each $j\in[n]$ the integer $e(j)\ge0$ denotes the number of $j$-cycles in the permutation $a$. Obviously $n=\sum_{j=1}^nje(j)$.   

We sometimes save space by omitting to write both $1^{e(1)}$ and also those $j^{e(j)}$ for which $e(j)=0$.\vspace{.5em}

\noindent{\bf Example.} If $a:=(0\, 1)(2\, 3)(4\, 5)(6\,7\,8\,9)\in$ Sym$(12)$ then Type$(a)=1^22^34^1$, which is to say $a\in 1^22^34^1$. But if we had prior knowledge that $a \in$ Sum$(12)$ then we 
might have written more tersely instead that $a\in 2^34^1$.\vspace{.5em}

\noindent{\bf Definition.} A sequence ${\bf s}:=\langle s_0,\ldots, s_{k-1}\rangle$ in Sym$(n)$ is {\em conjugacy invariant} (CI)  iff Prod$({\bf s})\subseteq$ Type$(\bigcirc{\bf s})$. 

\begin{prop}\label{CyclicConj} Let ${\bf b}:=\langle b_0,b_1,\ldots, b_{k-1}\rangle$ be a sequence in {\rm Sym}$(n)$. For ${\bf b}^{(i)}$ and $i\in k$, the expression  
${\bf b}^{(i)} :=\langle b_i,b_{i+1},\ldots, b_{k-2},b_{k-1},b_0,b_1,\ldots, b_{i-1}\rangle$, known as a ``cyclic conjugate'' of ${\bf b}$, satisfies $\bigcirc{\bf b}^{(i)}\simeq\bigcirc{\bf b}$. \end{prop}

\begin{proof} Let ${\bf p}:=\langle b_0,b_1,\ldots, b_{i-1}\rangle$ and ${\bf s}:=\langle b_i,b_{i+1},\ldots, b_{k-1}\rangle$. Then $\bigcirc{\bf b}^{(i)} = \bigcirc{\bf sp} = 
\bigcirc{\bf s}\circ\bigcirc{\bf p} \simeq \bigcirc{\bf p}\circ\bigcirc{\bf s} = \bigcirc{\bf ps} = \bigcirc{\bf b}$.  \end{proof}

\noindent{\bf Definition.} For ${\bf s} := \langle s_0,s_1,\ldots, s_{k-1}\rangle$ a sequence, ${\bf s}^{\rm R} :=\langle s_{k-1},\ldots,s_1,s_0\rangle$ is called the {\em reverse} of {\bf s}.\vspace{.5em}

\begin{prop}\label{Inverse} Let ${\bf t}:=\langle t_0,t_1,\ldots,t_{k-1}\rangle$ be a transpositional sequence in {\rm Sym}$(n)$. Then $\bigcirc({\bf t}^{\rm R})\simeq\bigcirc{\bf t}$. \end{prop}

\begin{proof} $\bigcirc({\bf t}^{\rm R}) = t_{k-1}\circ t_{k-2}\circ\cdots\circ t_1\circ t_0 = t_{k-1}^-\circ t_{k-2}^-\circ\cdots\circ t_1^-\circ t_0^- = (t_0\circ t_1\circ\cdots\circ t_{k-1})^- = 
(\bigcirc{\bf t})^- \simeq\bigcirc{\bf t}$ since $t_i^- = t_i$ when $t_i$ is a transposition. \end{proof}

Dan Franklin: By Proposition \ref{Inverse}, if {\bf s} a transpositional sequence and $f\in$ Prod$({\bf s})$, then $f^-\in$ Prod$({\bf s})$. \vspace{.5em}

\noindent{\bf Terminology.} When ${\bf w}:=\langle x\rangle$ is a length-one sequence, then $x$ may serve as a nickname for {\bf w}. If a sequence {\bf w} is of length $|{\bf w}|=k$ in $X$, and if {\bf w} 
occurs exactly $m$ times as a term in ${\bf r}=\langle {\bf w},{\bf w},\ldots, {\bf w}\rangle$, then we write ${\bf r := w}^{\beta(m)}$. That is, ${\bf w}^{\beta(m)}$ is the ``block'' consisting of exactly 
$m$ adjacent occurrences of {\bf w}. Thus {\bf r} has length $m$ when seen as a sequence in the set $\{{\bf w}\}$, but $|{\bf r}|=mk$ when {\bf r} is viewed as a sequence in $X$. 

Whereas ${\bf w}^{\beta(m)}$ denotes a sequence comprised of $m$ adjacent occurrences of the subsequence {\bf w}, the expression  
$(\bigcirc{\bf w})^m$ denotes the compositional product $\bigcirc{\bf w}\circ\bigcirc{\bf w}\circ\cdots\circ\bigcirc{\bf w}$ of $m$ adjacent occurrences of the permutation $\bigcirc{\bf w}$. 
That is to say, if ${\bf r}:={\bf w}^{\beta(m)}$ is a sequence in Sym$(n)$ then $\bigcirc{\bf r} = \bigcirc({\bf w}^{\beta(m)}) = (\bigcirc{\bf w})^m$. \vspace{.5em} 

Each sequence in Sym$(2)$ is both perm-complete and CI. If $|{\bf s}|<3$ for {\bf s} a sequence in Sym$(n)$ then {\bf s} is CI. However, for ${\bf s}$ in Sym$(n)$ with $n\ge3$ and with $|{\bf s}|\ge3$, the plot 
thickens.\vspace{.5em} 

When $a\not=b$ are vertices in a multigraph ${\cal G}$, the {\em multiplicity} in ${\cal G}$ of its multiedge $(a\,b)$ is the number $\mu_{\cal G}(a\,b)\ge0$ of simple edges in the bundle comprising that multiedge. 
Thus, when $\mu_{\cal G}(a\,b)=0$, there is no simple edge in ${\cal G}$ connecting $a$ with $b$. But, when $\mu_{\cal G}(a\,b)=1$, then the multiedge $(a\,b)$ is itself simple in ${\cal G}$. For 
${\bf u}\in1^{n-2}2^1$, the multiplicity $\mu_{\bf u}(a\,b)$ in {\bf u} of the transposition $(a\,b)$ as a term in {\bf u} equals $\mu_{{\cal T}({\bf u})}(a\,b)$.
\vspace{.5em}

\underline{Reminder}: $f\in1^{n-2}2^1$ says merely that $f$ is a transposition in Sym$(n)$. A multigraph ${\cal G}$ we call CI iff ${\cal G}$ is isomorphic to ${\cal T}({\bf t})$ for a CI sequence {\bf t} 
in $1^{n-2}2^1$. Without ado we will apply obviously corresponding terminology interchangeably to transpositional sequences and to isomorphs of transpositional multigraphs. 

\subsection{Conjugacy invariant transpositional sequences}

We proceed to identify the CI transpositional sequences {\bf u} in Sym$(n)$. It suffices to treat those such {\bf u} for which ${\cal T}({\bf u})$ is a connected multigraph on the vertex set $n$; this narrow focus is 
embodied in Theorem \ref{CCIT}.

\begin{thm}\label{OddTree} Let ${\cal T}({\bf u})$ be a multitree with no even-multiplicity multiedges, and none of whose vertices lie on more than one non-simple multiedge. Then 
{\rm Prod}$({\bf u})\subseteq n^1$.  \end{thm}
 
\begin{proof} We induce on $|{\bf u}|\ge n-1$. 

{\sf Basis Step}: The theorem is easily seen to hold when ${\cal T}({\bf u})$ is simple. Proofs occur in \cite{Denes} and in \cite{Silberger}.

{\sf Inductive Step}: Pick $k>n$. Suppose the theorem holds for all {\bf u} for which $|{\bf u}|\in\{n-1,n,\ldots,k-1\}$. Let $|{\bf u}|\in\{k-1,k\}$, and let {\bf u} satisfy the hypotheses of the theorem. 
 
Let $(x\,y)$  be a multiedge of ${\cal T}({\bf u})$ such that neither $x$ nor $y$ is an endpoint of any non-simple multiedge $(x'\,y')\not=(x\,y)$. Let {\bf v} be a sequence created by inserting into {\bf u} two 
additional occurrences, $(x\,y)_1$ and $(x\,y)_2$,  of the transposition $(x\,y)$. Thus ${\bf v} = \langle{\bf a},(x\,y)_1,{\bf b},(x\,y)_2,{\bf c}\rangle$ for some subsequences {\bf a}, {\bf b}, and {\bf c} of {\bf u} 
for which 
${\bf u = abc}$. If $|{\bf b}|=0$ then obviously $\bigcirc{\bf v}=\bigcirc{\bf u}\in n^1$.  So suppose that $|{\bf b}|>0$.

Let the first term of {\bf b} be $(t\,z)$. If $\{x,y\}=\{t,z\}$ then $(x\,y)_1\circ(t\,z)=\iota$, and so again $\bigcirc{\bf v}=\bigcirc{\bf u}$. But, if $\{x,y\}\cap\{t,z\}=\emptyset$, then 
$(x\,y)_1\circ(t\,z)=(t\,z)\circ(x\,y)_1$, and $(x\,y)_1$ will have migrated one space to the right in {\bf v} towards $(x\,y)_2$.  So take it that $y=t$ and that  $|\{x,y\}\cap\{y,z\}|=1$.  

Now, $(x\,y)_1\circ(y\,z)=(x\,z)\circ(x\,y)_1$. The tree ${\cal T}({\bf v})$ does not have the triangle $\{(x\,y)_1,(y\,z),(x\,z)\}$ as a subgraph. So the transposition $(x\,z)$ does not occur as a term in {\bf u}. 
Indeed, if {\bf v} satisfies the hypotheses of the lemma, then the multiplicity in {\bf v} of $(y\,z)$ is $1$, since the multiplicity of $(x\,y)$ in {\bf v} is greater than one. Thus the tree ${\cal T}({\bf w})$ is just the 
modification of ${\cal T}({\bf v})$ obtained by the replacement of the simple multiedge $(y\,z)$ of ${\cal T}({\bf v})$ by the simple multiedge $(x\,z)$. That is, {\bf w} has a single occurrence of the transposition 
$(x\,z)$ as a term but has no $(y\,z)$ terms, whereas {\bf v} has a single occurrence of $(y\,z)$  but has no occurrences of $(x\,z)$. Clearly {\bf w} also satisfies the hypotheses of the lemma, and $|{\bf w}|=
|{\bf v}|\in\{k+1,k+2\}$, since ${\bf w} = \langle{\bf a'},(x\,y)_1,{\bf b'},(x\,y)_2,{\bf c}\rangle$ where ${\bf a'}:=\langle{\bf a},(x\,z)\rangle$ and where ${\bf b'}$ is the sequence created by removing the 
leftmost term $(y\,z)$ of {\bf b}. So in this fashion too, $(x\,y)_1$  migrates one space rightward towards $(x\,y)_2$. The rightward migrations of $(x\,y)_1$  continue until $(x\,y)_1$ either abuts on $(x\,y)_2$ or 
on some occurrence of $(x\,y)$ to the left of $(x\,y)_2$. Thus the rightward migrations of $(x\,y)_1$ ultimately result in a sequence ${\bf w'}$ with $|{\bf w'}|\le k$ and for which $\bigcirc{\bf w'}\simeq\bigcirc{\bf u}$.  
Thus the inductive step is successful, and the theorem follows. \end{proof}

\begin{thm}\label{n=3} Let {\bf u} be a sequence in $1^12^1$ with $\bigcup${\rm Supp}$({\bf u})=3$. Then {\bf u} is {\rm CI} if and only if either $|{\bf u}|$ is odd or ${\cal T}({\bf u})$ is a 
multitree with at least one simple multiedge. \end{thm} 

\begin{proof} If $|{\bf u}|$ is odd then Prod$({\bf u})\subseteq 1^12^1$, and so {\bf u} is CI. For the rest of the proof we take $|{\bf u}|$ to be even.  

Let ${\cal T}({\bf u})$ be a tree with a simple multiedge $(0\, 1)$. If the multiedge $(1\,2)$ is simple too, then {\bf u} is CI. So take it that ${\bf u} := \langle(0\,1),(1\,2)^{\beta(2i+1)}\rangle$ for some $i\ge1$. 
Let ${\bf r}\in$ Seq$({\bf u})$. Then ${\bf r} = \langle(1\,2)^{\beta(j)},(0\,1),(1\,2)^{\beta(2i+1-j)}\rangle$ for some $j\in 2i+2$. So $\bigcirc{\bf r}=(1\,2)^j\circ(0\,1)\circ(1\,2)^{2i+1-j}$. If $j$ is even then 
$2i+1-j$ is odd, whence $\bigcirc{\bf r}= (0\,1)\circ(1\,2)=(0\,2\,1)\in3^1$, and if $j$ is odd then $2i+1-j$ is even, and so $\bigcirc{\bf r}=(1\,2)\circ(0\,1)=(0\,1\,2)\in3^1$. Therefore {\bf u} is CI in the event 
that ${\cal T}({\bf u})$ is a multitree, one of whose multiedges has multiplicity one.\vspace{.3em}

To establish the converse, we first consider the case where ${\cal T}({\bf u})$ is a multitree, and assume it has no simple multiedge. We can take it that ${\bf u}:=\langle(0\,1)^{\beta(i)},(1\,2)^{\beta(j)}\rangle$, 
where $i\ge2$ and $j\ge2$ and $i+j$ is even. The argument about this multitree obviously reduces to only two cases. 

\underline{Case}. $i=j=2$. Then $\bigcirc{\bf u}=\iota\restrict3\not\simeq(0\,1\,2)=((0\,1)\circ(1\,2))^2$.

\underline{Case}. $i=j=3$. Then $\bigcirc{\bf u}=(0\,2\,1)\not\simeq\iota\restrict3=(0\,2\,1)^3=((0\,1)\circ(1\,2))^3$.\vspace{.3em} 

Now suppose that ${\cal T}({\bf u})$ is a multitriangle with \ ${\bf u}:=\langle(0\,1)^{\beta(a)},(1\,2)^{\beta(b)},(2\,0)^{\beta(c)}\rangle$, where $1\le\min\{a,b,c\}$ and where $a+b+c$ is even. The argument 
again reduces to two cases.

\underline{Case}. $a=b=1$ and $c=2$. Then $\bigcirc{\bf u} = (0\,2\,1)\not\simeq\iota\restrict3=(0\,1)\circ(2\,0)\circ(1\,2)\circ(2\,0)$. 

\underline{Case}. $a=b=c=2$. Then $\bigcirc{\bf u}=\iota\restrict3\not\simeq(0\,1\,2)=(0\,1)\circ(1\,2)\circ(2\,0)\circ(1\,2)\circ(0\,1)\circ(2\,0)$.\vspace{.3em}

In all four cases we found an ${\bf r}\in$ Seq$({\bf u})$ with $\bigcirc{\bf r}\not\simeq\bigcirc{\bf u}$. So {\bf u} is not CI.  \end{proof}

Henceforth {\bf u} is a sequence in $1^{n-2}2^1$ for which ${\cal T}({\bf u})$ a connected multigraph whose vertex set is $n$. We have characterized the CI sequences for $n<4$. From now on, $n\ge4$. 
The {\bf u} we will be treating are of two sorts: \ {\sf One:} ${\cal T}({\bf u})$ is a multitree. \ {\sf Two:} ${\cal T}({\bf u})$ has a circuit subgraph. First we treat Sort One. \vspace{.5em}

By an  $m$-{\em twig} of a multigraph ${\cal G}$ we mean any multiplicity-$m$ multiedge $(v\,w)$, one of whose vertices has exactly one neighbor in ${\cal G}$. If $w$ is the only neighbor of the vertex $v$, 
then $v$ is the {\em leaf} of the multitwig.  

\begin{thm}\label{SufficiencyCI} Let the transpositional multitree ${\cal T}({\bf u})$ have exactly $b$ multiedges of even multiplicity, where ${\bf u}:=\langle u_0,u_1,\ldots, u_{k-1}\rangle$ is of length 
$|{\bf u}|:=k\ge3$ in $1^{n-2}2^1$ with $n\ge4$. Let the following two conditions hold: 

{\bf 3.6.1} \ No vertex lies on more than one non-simple multiedge. 

{\bf 3.6.2} \ Each even-multiplicity multiedge is a multitwig whose non-leaf vertex has exactly two neighbors. 

\noindent Then {\rm Prod}$({\bf u})\subseteq 1^b(n-b)^1$, and therefore {\bf u} is {\rm CI}. \end{thm}  

\begin{proof} Given $n\ge4$, we induce on $b\in\{0,1,\ldots,n-1\}$. 

{\sf Basis Step}: $b=0$. This is just Theorem \ref{OddTree}. 

{\sf Inductive Step}: Suppose, for each $m\in\{4,5,\ldots, n-1\}$ and each $X\subseteq n$ with $|X|=m$, that the theorem holds for every transpositional sequence {\bf t} in Sym$(X)$ for which ${\cal T}({\bf t})$ 
is a multitree  with vertex set $X$. By hypothesis, {\bf u} is a sequence in $1^{n-2}2^1$ that satisfies 3.6.1 and 3.6.2, where ${\cal T}({\bf u})$ has exactly $b$ even-multiplicity multitwigs, and where all of the 
non-multitwig multiedges of ${\cal T}({\bf u})$ are of odd multiplicity. Suppose $b\ge1$.

Let $(0\,1)$ be an even-multiplicity multitwig of ${\cal T}({\bf u})$ with leaf $0$. Let {\bf v} be the subsequence of {\bf u} obtained by removing all occurrences of $(0\,1)$ as terms in {\bf u}. Then ${\cal T}({\bf v})$ 
is a multitree on the set $X := n\setminus\{0\}$. Obviously ${\cal T}({\bf v})$ is a multitree that satisfies 3.6.1 and 3.6.2 and that has exactly $b-1$ even-multiplicity multitwigs. Since $|X|=n-1$, the inductive hypothesis 
implies that Prod$({\bf v})\subseteq 1^{b-1}((n-1)-(b-1))^1 = 1^{b-1}(n-b)^1$ and that {\bf v} is CI. By 3.6.2, the only multiedge of ${\cal T}({\bf u})$, other than $(0\,1)$, to share the vertex $1$ is a simple multiedge 
$(1\,x)$ of ${\cal T}({\bf v})$, and $(1\,x)$ is the only term of {\bf u} that fails to commute with $(0\,1)$. So $f\leftrightarrow f\cup(0)$ is a one-to-one matching Prod$({\bf v})\leftrightarrow$ Prod$({\bf u})$. It follows 
that  Prod$({\bf u})\subseteq 1^b(n-b)^1$ since Prod$({\bf v})\subseteq 1^{b-1}(n-b)^1$ by the inductive hypothesis. \end{proof}

Lemma \ref{n=3} gives necessary and sufficient conditions for {\bf u} to be CI when $n\le3$. Theorems \ref{OddTree} and  \ref{SufficiencyCI} give sufficient conditions for {\bf u} to be CI when $n\ge4$. We will show 
that those conditions are also necessary for $n\ge4$. The crux is to establish that, if the connected multigraph ${\cal T}({\bf u})$ on the vertex set $n\ge4$, fails to be a multitree satisfying both 3.5.1 and 3.5.2, then 
{\bf u} is not CI.  This project involves two subprojects: 

The first such subproject will show that, if ${\cal T}({\bf u})$ is a ``pathological'' multitree -- which is to say, one for which either 3.6.1 or 3.6.2 fails,  then {\bf u} cannot be CI. 

The last will show that, if $n\ge4$ and ${\cal T}({\bf u})$ has a circuit submultigraph, then again {\bf u} cannot be CI.\vspace{.5em}

For the balance of \S3, the expression {\bf u} will denote a sequence in $1^{n-2}2^1$ with $|{\bf u}|\ge n-1\ge3$, and such that the transpositional multigraph ${\cal T}({\bf u})$ is connected on the vertex set 
$n$. \vspace{1em}

\centerline{\sf Subproject One: \ To prove that, if ${\cal T}({\bf u})$ is a pathological multitree, then {\bf u} fails to be CI}\vspace{.5em}

We call a sequence ${\bf s}$ {\sl reduced} iff no entity occurs more than three times as a term in ${\bf s}$..

When at least one entity occurs as a term in a sequence ${\bf a}$ more than $3$ times, we may produce a reduced subsequence ${\bf c}$ of ${\bf a}$ by means of a string  of ``elementary reductions'':

 If $x$ occurs as a term more than $3$ times in {\bf a}, then a subsequence {\bf b} of {\bf a} is an {\em elementary reduction} of {\bf a} if {\bf b} is obtained by removing from {\bf a} two occurrences of $x$. The resulting 
such ${\bf b}$ is of length $|{\bf a}|-2$. 

A {\em reduction} of ${\bf a}$ is any reduced subsequence of ${\bf a}$ that results from a sequence of elementary reductions.  

Clearly each sequence ${\bf u}$ in $1^{n-2}2^1$ has a unique reduced subsequence. If ${\bf r}$ is a reduced subsequence of a transpositional sequence ${\bf s}$ then of course ${\rm Seq}({\bf r})$ is the set of all reduced  subsequences of elements in ${\rm Seq}({\bf s})$. 

We will employ the contrapositive version of the following obvious fact:

\begin{lem}\label{Reduced} A reduced subsequence of a {\rm CI} transpositional sequence is {\rm CI}. \end{lem}

We henceforth take it that all of our transpositional sequences are reduced, unless specified otherwise.

\begin{lem}\label{Untwig} If $\mu_{\bf u}(0\,1)=2$, but  if $(0\,1)$ is not a multitwig of the multitree ${\cal T}({\bf u})$, then {\bf u} fails to be {\rm CI}.  \end{lem}

\begin{proof} In the spirit developed earlier, ``$\mu_{\bf u}$'' is an abbreviation for ``$\mu_{{\cal T}({\bf u})}$''. 

Let $\mu_{\bf u}(0\,1)=2$ and the multiedge $(0\,1)$ of ${\cal T}({\bf u})$ not be a multitwig. Let {\bf v} be the subsequence of {\bf u} resulting from the removal from {\bf u} of its two occurrences 
of $(0\,1)$ as terms. \ ${\cal T}({\bf v})$ is the disjoint union ${\cal G}_0\,\dot{\cup}\,{\cal G}_1$ of two multitrees, each of which has a vertex set containing more than one vertex since the excised multiedge $(0\,1)$ 
of ${\cal T}({\bf u})$ was not a multitwig. So {\bf v} consists of two nonempty complementary subsequences ${\bf v}_0$ and ${\bf v}_1$, with $|{\bf v}_0|+|{\bf v}_1|=|{\bf v}|=|{\bf u}|-2\ge n-3\ge1$, and 
for which ${\cal G}_0={\cal T}({\bf v}_0)$ and ${\cal G}_1 = {\cal T}({\bf v}_1)$. That is to say, the terms of ${\bf v}_i$ are the simple edges of ${\cal G}_i$ for each $i\in2$. 

Let $f_0$ be the component of $\bigcirc{\bf v}_0$ such that $0\in$ supp$(f_0)$, and let $f_1$ be the component of $\bigcirc{\bf v}_1$ such that $1\in$ supp$(f_1)$, observing that neither $f_0$ nor $f_1$  is a $1$-cycle. 
Since our real concern is Seq$({\bf u})$, we can take it that ${\bf u}=\langle{\bf v}_0,(0\,1)^{\beta(2)},{\bf v}_1\rangle$ and that ${\bf v}={\bf v}_0{\bf v}_1$.  Of course, $0$ is a vertex in ${\cal G}_0$  and $1$ is a vertex in 
${\cal G}_1$. Then $f_0$ and $f_1$ are disjoint nontrivial cyclic components of the permutation $\bigcirc{\bf u}=\bigcirc{\bf v}=\bigcirc{\bf v}_0\bigcirc{\bf v}_1$.

Define ${\bf u'} := \langle(0\,1),{\bf v}_0,(0\,1),{\bf v}_1\rangle \in$ Seq$({\bf u})$. All of the components of $\bigcirc{\bf u}$ other than $f_0$ and $f_1$ are components also of $\bigcirc{\bf u'}$. 
So the only change made to $\bigcirc{\bf u}$ that creates $\bigcirc{\bf u'}$ is the replacement of the two components $f_0$ and $f_1$ with a new pair  $(0)$ and $h$, where $h$ is a cycle with $1\in$ supp$(h)$, 
and with $|h|=|f_0|+|f_1|-1$. So $\bigcirc{\bf u'} \not\simeq \bigcirc{\bf u}$, and hence {\bf u} is not CI. \end{proof}

Lemma \ref{Untwig} shows without loss of generality for $n\ge4$ that, if the transpositional multitree ${\cal T}({\bf u})$ has an even-multiplicity multiedge which is not a multitwig, then {\bf u} cannot be CI. 

\begin{cor}\label{TwoTwo} Let $\mu_{\bf u}(0\,1) =\mu_{\bf u}(1\,2)=2$. Then {\bf u} is not {\rm CI}.  \end{cor}

\begin{proof} Pretend that {\bf u} is CI. It follows by Lemma \ref{Untwig} that both of the multiedges $(0\,1)$ and $(1\,2)$ of the multitree ${\cal T}({\bf u})$ are multitwigs. Therefore, since $n\ge4$, there exists 
$x\in n\setminus 3$ for which $(1\,x)$ is a multiedge of ${\cal T}({\bf u})$. Let {\bf v} be the subsequence of {\bf u} that is produced by the removal from {\bf u} of both of the terms that are occurrences of the 
transposition $(0\,1)$ and both of the terms that are occurrences of $(1\,2)$. Then $|{\bf v}| = |{\bf u}|-4 \ge 5-4 = 1$. Let $f$ be the component of $\bigcirc{\bf v}$ with either $f = (1)$ or $1\in$ supp$(f)$. 
Observe that $\{0,2\}\cap\,$supp$(f) = \emptyset$. Since our interest lies in the sets Seq$({\bf u})$ and Seq$({\bf v})$, we can suppose that ${\bf u}=(0\,1)^{\beta(2)}(1\,2)^{\beta(2)}{\bf v}$. Of course $f$ is a 
component of $\bigcirc{\bf v}$ = $\bigcirc{\bf u}$. Defining ${\bf u'} := \langle\big((0\,1),(1\,2)\big)^{\beta(2)},{\bf v}\rangle \in$ Seq$({\bf u})$, we see that $\bigcirc{\bf u'} = (0\,1\,2)\circ\bigcirc{\bf v}$, a 
permutation which is identical to the permutation $\bigcirc{\bf u}$ in all component cycles that are disjoint from $3\,\cup\,$supp$(f)$. Where $(0), (2)$, and $f$ are components of $\bigcirc{\bf u}$, the permutation 
$\bigcirc{\bf u'}$ instead has the cycle $(0\,1\,2)\circ f$ of length $|f|+2$. Thus $\bigcirc{\bf u'}\not\simeq\bigcirc{\bf u}$, and so {\bf u} is not CI. \end{proof}

\begin{cor}\label{TwoThree} Let  ${\bf u} := (0\,1)^{\beta(2)}(1\,2)^{\beta(3)}{\bf v}$ and {\bf v} be sequences in $1^{n-2}2^2$, where neither $(0\,1)$ nor $(1\,2)$ is a term in {\bf v}. Then {\bf u} is not {\rm CI}.  
\end{cor}

\begin{proof} Assume that {\bf u} is CI. By Lemma \ref{Untwig}, the multiedge $(0\,1)$ of the multitree ${\cal T}({\bf u})$ is a multitwig of ${\cal T}({\bf u})$. So there is a component $f$ of the permutation 
$\bigcirc{\bf v}$ to which exactly one of the following two cases applies.

\underline{Case}: Either  $2\in$ supp$(f)$ or $f = (2)$,  and $\bigcirc{\bf u} = (0\,1)^2\circ(1\,2)^3\circ\bigcirc{\bf v} = (1\,2)\circ\bigcirc{\bf v}$. So $f_2 := (1\,2)\circ f$ is a cyclic component of $\bigcirc{\bf u}$. 
Note: \ $|f_2|=|f|+1$, since the point $1$ is incorporated into the cycle $f$ in order to create $f_2$.  [Paradigm example: When $f := (2\,3\,4)$ then $f_2 = (1\,2)\circ f = (1\,2)\circ(2\,3\,4)=(1\,3\,4\,2)$.] Define 
${\bf u}_2 := \langle(1\,2),(0\,1),(1\,2),(0\,1),(1\,2),{\bf v}\rangle \in$ Seq$({\bf u})$. Then $\bigcirc{\bf u}_2 = (0\,1)\circ\bigcirc{\bf v} = (0\,1)\bigcirc{\bf v}$, and $f$ is a cyclic component of $\bigcirc{\bf u}_2$. 

\underline{Case}: Either $1\in{\rm supp}(f)$ or $f=(1)$, and $\bigcirc{\bf u} = (1\,2)\circ\bigcirc{\bf v}$. So $f_1 = (1\,2)\circ f$ is a cyclic component of $\bigcirc{\bf u}$. Let 
${\bf u}_1 := \langle(0\,1),(1\,2)^{\beta(3)},(0\,1),{\bf v}\rangle$. Then $\bigcirc{\bf u}_1 = (0\,2)\circ\bigcirc{\bf v} = (0\,2)\bigcirc{\bf v}$, and $f$ is a component of $\bigcirc{\bf u}_1$. But $|f_1|=|f|+1$.
  
We showed, for each $i\in2$, that $|f_i|=|f|+1$. Moreover, $\bigcirc{\bf u}$ has one more $1$-cycle and one fewer $2$-cycles than $\bigcirc{\bf u}_i$ has, while all other cyclic components of $\bigcirc{\bf u}_i$ 
are the same as those of $\bigcirc{\bf u}$.  Hence $\bigcirc{\bf u}_i\not\simeq\bigcirc{\bf u}$ in both Cases. Thus our assumption fails. Therefore {\bf u} is not CI.\end{proof}

\begin{lem}\label{ThreeThree} Let ${\bf u} := (0\,1)^{\beta(3)}(1\,2)^{\beta(3)}{\bf v}_0{\bf v}_1{\bf v}_2$ where neither $(0\,1)$ nor $(1\,2)$ is a term in the sequence ${\bf v}_0{\bf v}_1{\bf v}_2$, and where 
no vertex of ${\cal T}({\bf v}_i)$ is a vertex in ${\cal T}({\bf v}_j)$ if $i\not=j$. Then {\bf u} is not {\rm CI}. \end{lem}

\begin{proof} We can suppose at least one of the three subsequences ${\bf v}_i$ to be nonvacuous since $n\ge4$. Each ${\cal T}({\bf v}_i)$ is a (possibly one-vertex) submultigraph of 
the transpositional multitree ${\cal T}({\bf u})$, where for each $i\in3$ we are given that $i$ is a vertex in ${\cal T}({\bf v}_i)$. Now, $\bigcirc{\bf u} = (0\,2\,1)\circ\bigcirc{\bf v}_0\bigcirc{\bf v}_1\bigcirc{\bf v}_2$.  

For each $i\in3$, let $f_i$ be the component of $\bigcirc{\bf v}_i$ for which either $i\in$ supp$(f_i)$ or $f_i=(i)$. Then $\bigcirc{\bf u}$ has a cyclic component $f$ of length $|f|=|f_0|+|f_1|+|f_2|$ with 
$3\subseteq$ supp$(f)$. 

Let ${\bf u'} := \langle (1\,2),(0\,1),{\bf v}_0,(1\,2),(0\,1),{\bf v}_1,(1\,2),(0\,1){\bf v}_2 \rangle$. So $\bigcirc{\bf u'} = (0\,1\,2)\circ\bigcirc{\bf v}_0\circ(0\,1\,2)\circ\bigcirc{\bf v}_1\circ(0\,1\,2)\circ\bigcirc{\bf v}_2$ lacks the cyclic component $f$  of $\bigcirc{\bf u}$, but in place of $f$ it has the three cycles $f_0,   f_1$ and $ f_2$, and otherwise the cycles of $\bigcirc{\bf u'}$ are identical to those of $\bigcirc{\bf u}$. So $\bigcirc{\bf u'}\not\simeq\bigcirc{\bf u}$  although ${\bf u'}\in$ Seq$({\bf u})$. Therefore {\bf u} is not CI.        \end{proof} 

\begin{cor}\label{Fork} Let $\mu_{\bf u}(0\,1)=2$, let $\mu_{\bf u}(1\,2) = 1 = \mu_{\bf u}(1\,3)$, and let ${\bf u} := \langle(0\,1)^{\beta(2)},(1\,2),(1\,3),{\bf v}_2{\bf v}_3\rangle$, where the two submultigraphs 
 ${\cal T}({\bf v}_2)$ and ${\cal T}({\bf v}_3)$ of ${\cal T}({\bf u})$ are disjoint. Then {\bf u} is not {\rm CI}.  \end{cor}

\begin{proof}  Assume that {\bf u} is CI. By Lemma \ref{Untwig}, the multiedge $(0\,1)$ of ${\cal T}({\bf u})$ is a multitwig with leaf $0$. We take it that $f_2$ is a cyclic component of 
$\bigcirc{\bf v}_2$ for which either $2\in$ supp$(f_2)$ or $f_2=(2)$, and likewise that $f_3$ is a cyclic component of $\bigcirc{\bf v}_3$ for which either $3\in$ supp$(f_3)$ or $f_3=(3)$. Now 
$\bigcirc{\bf u} = (0)fg$, where $f$ is a cycle incorporating the point $1$ together with the points in $f_2$ and $f_3$ into a single cycle of consequent length $|f|=1+|f_2|+|f_3|$, where $g$ is a permutation 
that involves the points in $\{4,5,\ldots,n-1\}$ which occur neither in $f_2$ nor in $f_3$. On the other hand, defining  ${\bf u'} := \langle(0\,1),(1\,2),(0\,1),(1\,3),{\bf v}_2{\bf v}_3\rangle \in$ Seq$({\bf u})$, we 
find that $\bigcirc{\bf u'} = f'_2f'_3g$, where $f'_2$ is a cycle of length $|f'_2|=1+|f_2|$ that incorporates together the point $0$ and the points in the cycle $f_2$, and where $f'_3$ is a cycle of length 
$|f'_3|=1+|f_3|$ that incorporates together the point $1$ and the points in the cycle $f_2$. So $\bigcirc{\bf u'}\not\simeq\bigcirc{\bf u}$, violating our assumption that {\bf u} is CI. \end{proof}

Subproject One is completed. We summarize it in the following immediate conjunction of Lemma \ref{Untwig}, Corollaries \ref{TwoTwo} and \ref{TwoThree}, Lemma \ref{ThreeThree}, and Corollary \ref{Fork}:

\begin{thm}\label{NecessaryCI} Let ${\cal T}({\bf u})$ be a multitree with $n\ge4$. Then ${\bf u}$ is CI if and only if it satisfies 3.6.1 and 3.6.2,  \end{thm}
\vspace{1em}

\centerline{\sf Subproject Two: \ Proving for $n\ge4$ that, if {\bf u} is CI, then ${\cal T}({\bf u})$ has no circuits} \vspace{.5em} 

For $n\ge4$, our focus now is upon those sequences {\bf u} in $1^{n-2}2^1$ for which the transpositional multigraph ${\cal T}({\bf u})$ is connected on the vertex set $n$, but is not a multitree; instead, 
${\cal T}({\bf u})$  has at least one circuit subgraph. We will now provide, some convenient additional terminological background.\vspace{.5em}

Although we write a sequence usually between pointy brackets -- e.g., $\langle x_0,x_1,\ldots,x_{k-1}\rangle$ -- with its terms separated by commas, when ambiguity is not at issue, we may write it with (some 
or all of) its terms concatenated (i.e., without commas.) However, when $f$ and $g$ are permutations whose supports are distinct, we have been writing $f\circ g$ as $fg$ in order to indicate this disjointness. 
Context will make it clear whether an expression denotes disjoint permutations instead of concatenated sequences. 

When a sequence is of length one, we call its single term {\em primitive}.

A few specific sequences, to which we frequently refer, we will honor with the adjective {\em basic}. 

Thus far, all of the sequences we have treated in detail are {\em permutational} sequences; their terms either are permutations or are characters denoting sequences of permutations. Indeed, almost all of our 
permutational sequences are {\em transpositional}: Their terms are either transpositions or characters denoting sequences of transpositions. Non-basic permutational sequences get lower-case bold-face 
Latin-letter names.
 
For the present subproject, when $n\ge4$, we shall have recourse to two basic tranpositional sequences, $\sigma(n)$ and $\tau(n)$. But we shall use number (integer) sequences as well; our basic number sequence
is written $\nu(n)$. Number sequences other than $\nu(n)$ will usually receive lower-case Latin letter designations. 

\begin{defn}\label{Sequences} \  $\tau(n) := \langle(0\,1),(1\,2),\ldots,(n-2\,\,n-1)\rangle$ \ and \ $\sigma(n) := \langle\tau(n),(n-1\,\,0)\rangle$. Also,  \ $\nu(n) := \langle 0,1,\ldots,n-1\rangle$. \end{defn}

Of course ${\cal T}(\sigma(n))$ is a simple circuit multigraph on $n$ vertices, with $n\ge4$ understood, and ${\cal T}(\tau(n))$ is the simple branchless multitree resulting from the removal of the simple multiedge 
$(n-1\,\,0)$ from ${\cal T}(\sigma(n))$.\vspace{.5em}

Before we treat circuit-containing connected multigraphs with $n\ge4$, we recall that Theorem \ref{n=3} settles the case for $n\le3$. Now, for $n\ge4$, we show that, if the transpositional multigraph 
${\cal T}({\bf u})$ contains a 4-vertex simple subgraph which is a triangle sprouting a twig, then {\bf u} is not CI. Remember: \ $4 := \{0,1,2,3\}$.

\begin{thm}\label{Four} Let $n\ge4$, and let ${\bf u}$ be a sequence in $1^{n-2}2^1$ which has ${\bf h} := \langle(0\,1),(1\,2),(0\,2),(0\,3)\rangle$ as a subsequence\footnote{We write ${\bf u\setminus h}$ to designate the subsequence of ${\bf u}$ obtained by removing from ${\bf u}$ its subsequence ${\bf h}$.} Then {\bf u} is not {\rm CI}. \end{thm} 

\begin{proof} Let $W := \{c: c$ is a cyclic component of $\bigcirc({\bf u \setminus h})$ with $4\,\cap\,$supp$(c)\not=\emptyset\}$.  Let $w\in{\rm Sym}(n)$ be the permutation having $W$ as its set of cyclic components.  
It suffices to show that $\bigcirc{\bf p}\circ w\not\simeq\bigcirc{\bf h}\circ w$ for some ${\bf p}\in{\rm Seq}({\bf h})$. There are five cases to treat.\vspace{.5em}

{\sf Case 1:} \ $|4\,\cap\,$supp$(c)|=1$ for every $c\in W$. Then $W = \{(0\,{\rm s}_0),(1\,{\rm s}_1),(2\,{\rm s}_2),(3\,{\rm s}_3)\}$ for some sequences ${\rm s}_i$ in $\{4,5,\ldots,n-1\}$. Consider the following three 
rearrangements ${\bf p}_i\in$ Seq$({\bf h})$:\vspace{.3em} 

\centerline{${\bf p}_1 := \langle(0\,2),(0\,3),(1\,2),(0\,1)\rangle$ \ \ and \ \ ${\bf p}_2 := \langle(0\,2),(0\,1),(0\,3),(1\,2)\rangle$ \ \ and \ \ ${\bf p}_3 := \langle(0\,2),(0\,3),(0\,1),(1\,2)\rangle$}\vspace{.3em} 

\noindent Then $\bigcirc{\bf p}_1 = (0)(1\,2\,3)$ \ and \ $\bigcirc{\bf p}_2 = (0\,1\,3)(2)$, \ and \ $\bigcirc{\bf p}_3 = (0\,1)(2\,3)$. Consequently $\bigcirc{\bf p}_1\circ w = (1\,2\,3)\circ w = 
(1\,2\,3)\circ(0\,{\rm s}_0)(1\,{\rm s}_1)(2\,{\rm s}_2)(3\,{\rm s}_3) = (0\,\,{\rm s}_0)(1\,\,{\rm s}_2\,2\,\,{\rm s}_3\,3\,\,{\rm s}_1)$. Similarly, $\bigcirc{\bf p}_2\circ w = (0\,1\,3)\circ w = 
(0\,\,{\rm s}_1\,1\,\,{\rm s}_3\,3\,\,{\rm s}_0)(2\,\,{\rm s}_2)$ \ and \ $\bigcirc{\bf p}_3\circ w = (0\,1)(2\,3)\circ w = (0\,\,{\rm s}_1\,1\,\,{\rm s}_0)(2\,\,{\rm s}_3\,3\,\,{\rm s}_2)$. Summarizing, we have 
that\vspace{.3em}

\centerline{$\bigcirc{\bf p}_1\circ w = (0\,\,{\rm s}_0)(1\,\,{\rm s}_2 \, 2\,\,{\rm s}_3\,3\,\,{\rm s}_1)$ and $\bigcirc{\bf p}_2\circ w =(0\,\,{\rm s}_1\,1\,\,{\rm s}_3\,3\,\,{\rm s}_0)(2\,\,{\rm s}_2) $ and 
$\bigcirc{\bf p}_3\circ w = (0\,\,{\rm s}_1\,1\,\,{\rm s}_0)(2\,\,{\rm s}_3\,3\,\,{\rm s}_2)$.}\vspace{.3em} 

In order to establish that {\bf u} is not CI, it suffices to show that these three permutations $\bigcirc{\bf p}_i\circ w$ are not members of the same one conjugacy class. Observe that, for each $i\in\{1,2,3\}$, the 
permutation $\bigcirc{\bf p}_i\circ w$ has exactly two cyclic components, $a_i$ and $b_i$. To argue by contradiction, we assume the multiset equalities $\{|a_1|,|b_1|\}=\{|a_2|,|b_2|\}=\{|a_3|,|b_3|\}$. 
Spelled out, these multiset equalities are\vspace{.3em}

\centerline{$\{|(0\,\,{\rm s}_0)|,|(1\,\,{\rm s}_2\,2\,\,{\rm s}_3\,3\,\,{\rm s}_1)|\} = \{|(0\,\,{\rm s}_1\,1\,\,{\rm s}_3\,3\,\,{\rm s}_0)|,|(2\,\,{\rm s}_2)|\} = 
\{|(0\,\,{\rm s}_1\,1\,\,{\rm s}_0)|,|(2\,\,{\rm s}_3\,3\,\,{\rm s}_2)|\}$, \ whence }\vspace{.3em} 

\centerline{\{$1+|{\rm s}_0|,3+|{\rm s}_2|+|{\rm s}_3|+|{\rm s}_1|\} = \{3+|{\rm s}_1|+|{\rm s}_3|+|{\rm s}_0|,1+|{\rm s}_2|\} = \{2+|{\rm s}_1|+|{\rm s}_0|,2+|{\rm s}_3|+|{\rm s}_2|\}$.  }\vspace{.3em}

\noindent Since $1+|{\rm s}_0| < 3+|{\rm s}_1|+|{\rm s}_3|+|{\rm s}_0|$, the equality \{$1+|{\rm s}_0|,3+|{\rm s}_2|+|{\rm s}_3|+|{\rm s}_1|\} = \{3+|{\rm s}_1|+|{\rm s}_3|+|{\rm s}_0|,1+|{\rm s}_2|\}$ 
implies that $1+|{\rm s}_0| = 1+|{\rm s}_2|$; so $|{\rm s}_0|=|{\rm s}_2|$. Therefore, \{$1+|{\rm s}_0|,3+|{\rm s}_2|+|{\rm s}_3|+|{\rm s}_1|\} =  \{2+|{\rm s}_1|+|{\rm s}_0|,2+|{\rm s}_3|+|{\rm s}_2|\}$  
implies that $1+|{\rm s}_0| = 2+|{\rm s}_3|+|{\rm s}_2|$ since $1+|{\rm s}_0| < 2+|{\rm s}_1|+|{\rm s}_0|$. Hence, $1+|{\rm s}_0| = 2+|{\rm s}_3|+|{\rm s}_0|$, forcing us to the impossibility $|{\rm s}_3|=-1$. 
So the assumed three multiset equalities cannot hold simultaneously. Therefore $\bigcirc{\bf p}_i\circ w\not\simeq\bigcirc{\bf h}\circ w$ for at least one $i\in\{1,2,3\}$. We infer that {\bf u} is not CI in the Case 1 
situation.\vspace{.5em}

In the remaining four cases, $\psi$ denotes an arbitrary element in Sym$(4)$.\vspace{.5em}

{\sf Case 2:} \ $W = \{w\}$, and $4\subseteq$ supp$(w)$. That is, $w = \big(\psi(0)\,\,{\rm s}_{\psi(0)}\,\,\psi(1)\,\,{\rm s}_{\psi(1)}\,\,\psi(2)\,\,{\rm s}_{\psi(2)}\,\,\psi(3)\,\,{\rm s}_{\psi(3)}\big)\in n^1$, where the 
four s$_{\psi(i)}$ are number sequences, the family of whose nonempty term sets is a partition of the set $n\setminus 4 := \{4,5,\ldots,n-1\}$. Consider the subset $\{{\bf p}_4,{\bf p}_5\}\subseteq$ Seq$({\bf h})$ given by \vspace{.3em}

\centerline{${\bf p}_4 := \langle(0\,1),(0\,3),(0\,2),(1\,2)\rangle$ \  \ and \  \ ${\bf p}_5 := \langle(1\,2),(0\,1),(0\,3),(0\,2)\rangle$. }\vspace{.3em}

\noindent Then $\bigcirc{\bf p}_4 = (0\,2)(1\,3)$ and $\bigcirc{\bf p}_5 = (0\,1)(2\,3)$ and $\bigcirc{\bf h} = (0\,1)\circ(1\,2)\circ(0\,2)\circ(0\,3) = (0\,3)(1\,2)$. Observe that 
$\{\bigcirc{\bf p}_4,\bigcirc{\bf p}_5,\bigcirc{\bf h}\} = 2^2\subset$ Alt$(4)$. Hence there exists $\{{\bf p},{\bf q}\}\subseteq$ Seq$({\bf h})$ for which $\bigcirc{\bf p} = \big(\psi(0)\,\,\psi(1)\big)\big(\psi(2)\,\,\psi(3)\big)$ 
and for which $\bigcirc{\bf q} = \big(\psi(0)\,\,\psi(2)\big)\big(\psi(1)\,\,\psi(3)\big)$.  By straightforward computation we now obtain that 
\[\bigcirc{\bf p}\circ w = \big(\psi(0)\,\,{\rm s}_{\psi(1)}\,\,\psi(2)\,\,{\rm s}_{\psi(3)}\big)\big(\psi(1)\,\,{\rm s}_{\psi(0)}\big)\big(\psi(3)\,\,{\rm s}_{\psi(2)}\big)\, \not\simeq\,\]
\[ \big(\psi(0)\,\,{\rm s}_{\psi(2)}\,\,\psi(3)\,\,{\rm s}_{\psi(1)}\,\,\psi(2)\,\,{\rm s}_{\psi(0)}\,\,\psi(1)\,\,{\rm s}_{\psi(3)}\big) = \bigcirc{\bf q}\circ w.\] Thus we infer that {\bf u} fails to be CI in the Case 2 
situation.\vspace{.8em} 

{\sf Case 3:} \ $W = \{c_1,c_2\}$ where $|4\,\cap\,$supp$(c_i)|=2$ for each $i\in\{1,2\}$. So this time we can write $w = c_1c_2 = 
\big(\psi(0)\,\,{\rm s}_{\psi(0)}\,\,\psi(1)\,\,{\rm s}_{\psi(1)}\big)\big(\psi(2)\,\,{\rm s}_{\psi(2)}\,\,\psi(3)\,\,{\rm s}_{\psi(3)}\big)$. As in Case 2, here too we can provide $\{{\bf p,q}\}\subseteq$ Seq$({\bf h})$, for which 
$\bigcirc{\bf p} = \big(\psi(0)\,\,\psi(1)\big)\big(\psi(2)\,\,\psi(3)\big)$ and for which $\bigcirc{\bf q} = \big(\psi(0)\,\,\psi(2)\big)\big(\psi(1)\,\,\psi(3)\big)$.  We compute that 
\[\bigcirc{\bf p}\circ w = \big(\psi(0)\,\,{\rm s}_{\psi(1)}\big)\big(\psi(1)\,\,{\rm s}_{\psi(0)}\big)\big(\psi(2)\,\,{\rm s}_{\psi(3)}\big)\big(\psi(3)\,\,{\rm s}_{\psi(2)}\big)\,\not\simeq\] 
\[\big(\psi(0)\,\,{\rm s}_{\psi(2)}\,\,\psi(3)\,\,{\rm s}_{\psi(1)}\big)\big(\psi(1)\,\,{\rm s}_{\psi(3)}\,\,\psi(2)\,\,{\rm s}_{\psi(0)}\big) = \bigcirc{\bf q}\circ w.\]

\noindent Thus in the situation of Case 3 we again find that {\bf u} is not CI.\vspace{.5em}

{\sf Case 4:} \ $W = \{c_1,c_2,c_3\}$ with $c_1 := \big(\psi(0)\,\,{\rm s}_{\psi(0)}\,\,\psi(1)\,\,{\rm s}_{\psi(1)}\big)$ and $c_2 :=\big(\psi(2)\,\,{\rm s}_{\psi(2)}\big)$ and $c_3 = \big(\psi(3)\,\,{\rm s}_{\psi(3)}\big)$. 
So $w = c_1c_2c_3 = \big(\psi(0)\,\,{\rm s}_{\psi(0)}\,\,\psi(1)\,\,{\rm s}_{\psi(1)}\big)\big(\psi(2)\,\,{\rm s}_{\psi(2)}\big)\big(\psi(3)\,\,{\rm s}_{\psi(3)}\big)$.  Let {\bf p} and {\bf q} be as in Cases 2 and 3. Then 
\[\bigcirc{\bf p}\circ w = \big(\psi(0)\,\,{\rm s}_{\psi(1)}\big)\big(\psi(1)\,\,{\rm s}_{\psi(0)}\big)\big(\psi(2)\,\,{\rm s}_{\psi(3)}\,\,\psi(3)\,\,{\rm s}_{\psi(2)}\big)\,\not\simeq\]
\[\big(\psi(0)\,\,{\rm s}_{\psi(2)}\,\,\psi(2)\,\,{\rm s}_{\psi(0)}\,\,\psi(1)\,\,{\rm s}_{\psi(3)}\,\,\psi(3)\,\,{\rm s}_{\psi(1)}\big) = \bigcirc{\bf q}\circ w.\] Thus {\bf u} fails to be CI in the Case 4 situation as well. 
\vspace{.5em}

{\sf Case 5:} \ $W = \{c_1,c_2\}$ with $c_1 = \big(\psi(0)\,\,{\rm s}_{\psi(0)}\big)$ and $c_2 = \big(\psi(1)\,\,{\rm s}_{\psi(1)}\,\,\psi(2)\,\,{\rm s}_{\psi(2)}\,\,\psi(3)\,\,{\rm s}_{\psi(3)}\big)$. That is to say, 
$w =  \big(\psi(0)\,\,{\rm s}_{\psi(0)}\big) \big(\psi(1)\,\,{\rm s}_{\psi(1)}\,\,\psi(2)\,\,{\rm s}_{\psi(2)}\,\,\psi(3)\,\,{\rm s}_{\psi(3)}\big)$.  The $\bigcirc{\bf r}_i$ of the following  six ${\bf r}_i\in{\rm Seq}({\bf h})$ comprise the conjugacy class, $1^13^1 \subset{\rm Alt}(4)$, the six  possible $3$-cycles: 
\[{\bf r}_1 := \langle(0\,\,1),(1\,\,2),(0\,\,3),(0\,\,2)\rangle\quad\mbox{for which we compute that}\quad \bigcirc{\bf r}_1 = (0)(1\,\,3\,\,2)\] 
\[{\bf r}_2 := \langle(0\,\,2),(1\,\,2),(0\,\,3),(0\,\,1)\rangle\quad\mbox{for which we compute that}\quad \bigcirc{\bf r}_2 = (0)(1\,\,2\,\,3)\]
\[{\bf r}_3 := \langle(1\,\,2),(0\,\,3),(0\,\,2),(0\,\,1)\rangle\quad\mbox{for which we compute that}\quad \bigcirc{\bf r_3} = (0\,\,3\,\,2)(1)\]
\[{\bf r}_4 := \langle(0\,\,1),(0\,\,2),(1\,\,2),(0\,\,3)\rangle\quad\mbox{for which we compute that}\quad \bigcirc{\bf r}_4 = (0\,\,2\,\,3)(1)\]
\[{\bf r}_5 := \langle(1\,\,2),(0\,\,3),(0\,\,1),(0\,\,2)\rangle\quad\mbox{for which we compute that}\quad \bigcirc{\bf r}_5 = (0\,\,3\,\,1)(2)\]
\[{\bf r}_6 := \langle(1\,\,2),(0\,\,1),(0\,\,2),(0\,\,3)\rangle\quad\mbox{for which we compute that}\quad \bigcirc{\bf r}_6 = (0\,\,1\,\,3)(2)\]

\underline{Subcase}: \ $\psi(0) \not= 3$. Then $\bigcirc{\bf p} = \big(\psi(0)\big)\big(\psi(1)\,\,\psi(2)\,\,\psi(3)\big)$ and $\bigcirc{\bf q} = \big(\psi(0)\big)\big(\psi(1)\,\,\psi(3)\,\,\psi(2)\big)$
for some ${\{\bf p},{\bf q}\}\subseteq\{{\bf r}_i: 1\le i\le6\}$. 

We compute that $\bigcirc{\bf p}\circ w = 
\big(\psi(0)\big)\big(\psi(1)\,\,\psi(2)\,\,\psi(3)\big)\circ\big(\psi(0)\,\,{\rm s}_{\psi(0)}\big) \big(\psi(1)\,\,{\rm s}_{\psi(1)}\,\,\psi(2)\,\,{\rm s}_{\psi(2)}\,\,\psi(3)\,\,{\rm s}_{\psi(3)}\big) = 
\big(\psi(0)\,\,{\rm s}_{\psi(0)}\big)\big(\psi(1)\,\,{\rm s}_{\psi(2)}\,\,\psi(3)\,\,{\rm s}_{\psi(1)}\,\,\psi(2)\,\,{\rm s}_{\psi(3)}\big) \not\simeq 
\big(\psi(0)\,\,{\rm s}_{\psi(0)}\big)\big(\psi(1)\,\,{\rm s}_{\psi(3)}\big)\big(\psi(2)\,\,{\rm s}_{\psi(1)}\big)\big(\psi(3)\,\,{\rm s}_{\psi(2)}\big) = \bigcirc{\bf q}\circ w$. Thus we see that {\bf u} is not CI 
in the situation of Case 5 where also $\psi(0) \not= 3$. \vspace{.5em} 

\underline{Subcase}: \ $\psi(0) = 3$. There are two subsubcases, which are: 

\ \ \ \ \ \underline{i}: \ \ $w := (0\,\,{\rm s}_0\,\,1\,\,{\rm s}_1\,\,2\,\,{\rm s}_2)(3\,\,{\rm s}_3)$.  

\ \ \ \ \ \underline{ii}: \ $w := (0\,\,{\rm s}_0\,\,2\,\,{\rm s}_2\,\,1\,\,{\rm s}_1)(3\,\,{\rm s}_3)$.  

\noindent We will show that the theorem holds for Subsubcase \underline{i}, but omit the similar proof for Subsubcase \underline{ii}.\vspace{.3em} 

We use rearrangements ${\bf v}_1 := {\bf h}$ and ${\bf v}_2 := \langle(0\,1),(0\,3),(0\,2),(1\,2)\rangle$ and ${\bf v}_3 := \langle(0\,2),(0\,3),(1\,2),(0\,1)\rangle$ of {\bf h}, noting first that 
$\bigcirc{\bf v}_1 = (0\,3)(1\,2)$, that $\bigcirc{\bf v}_2 = (0\,2)(1\,3)$, and that $\bigcirc{\bf v}_3 = (0)(1\,2\,3)$, and hence that 
$\bigcirc{\bf v}_1\circ w = (0\,\,{\rm s}_3\,\,3\,\,{\rm s}_0\,\,1\,\,{\rm s}_2)(2\,\,{\rm s}_1)$ and $\bigcirc{\bf v}_2\circ w = (0\,\,{\rm s}_2)(1\,\,{\rm s}_3\,\,3\,\,{\rm s}_1\,\,2\,\,{\rm s}_0)$ and 
$\bigcirc{\bf v}_3\circ w = (0\,\,{\rm s}_0\,\,1\,\,{\rm s}_2)(2\,\,{\rm s}_3\,\,3\,\,{\rm s}_1)$. 

Assume that $\bigcirc{\bf v}_1\circ w \simeq \bigcirc{\bf v}_2\circ w \simeq \bigcirc{\bf v}_3\circ w$. Then the following three multiset equalites must hold:  
\[ \{3+|{\rm s}_3|+|{\rm s}_0|+|{\rm s}_2|,1+|{\rm s}_1|\} = \{1+|{\rm s}_2|,3+|{\rm s}_3|+|{\rm s}_1|+|{\rm s}_0|\}  = \{2+|{\rm s}_0|+|{\rm s}_2|,2+|{\rm s}_3|+|{\rm s}_1|\}\]

\noindent Since $1+|{\rm s}_2| < 3+|{\rm s}_3|+|{\rm s}_0|+|{\rm s}_2|$, the equality of the first two multisets implies that $1+|{\rm s}_1| = 1+|{\rm s}_2|$, whence $|{\rm s}_1| = |{\rm s}_2|$. 
Since $1+|{\rm s}_1| < 2+|{\rm s}_3|+|{\rm s}_1|$, the equality of the first and third multisets therefore implies that $1+|{\rm s}_1| = 2+|{\rm s}_0|+|{\rm s}_2| = 2+|{\rm s}_0|+|{\rm s}_1|$, 
whence $0 = 1+|{\rm s}_0|$, which entails the impossibility $|{\rm s}_0| = -1$.

So {\bf u} fails to be CI in Case 5 as well. Since the five Cases are exhaustive, the theorem is proved. \end{proof}

Theorem \ref{Four} gives us that, if $n\ge4$ and if ${\cal T}({\bf u})$ is a transpositional multigraph containing a triangular subgraph, then {\bf u} is not CI. The remainder of \S3 is devoted mainly to 
generalizing the proof of Theorem \ref{Four} in order to establish, for $n\ge4$, that no connected transpositional multigraph on $n$ vertices is CI if it contains  a circuit subgraph on more than three vertices. 
To this purpose it is useful to describe those sequences {\bf g} in $1^{n-2}2^1$ for which ${\cal T}({\bf g})$ is itself a circuit. The following three lemmas do so.

Recall the basic transpositional sequence $\tau(n) := \langle(0\,1),(1\,2),\ldots,(n-3\,\,n-2),(n-2\,\,n-1)\rangle$; that ${\rm s^R}$ is the backward spelling of the sequence \ s: \ and that, when \ a \ is a subsequence 
of a sequence \ b, \ then \ b $\setminus$ a \ denotes the subsequence of \ b \ that is complementary to \ a, \ as per Footnote 7. 

When ${\bf s}$ is a sequence, we write $x<_{\bf s} y$ to indicate that $x$ precedes $y$ as a term in ${\bf s}$. 

\begin{lem}\label{Numbers1} For $n\ge3$, let {\bf g} be any rearrangement of $\tau(n)$. Then $\bigcirc{\bf g} = (0\,\,{\rm p}\,\,n-1\,\,{\rm q}) \in n^1$ for some subsequence {\rm p} of 
$\langle1,2,\ldots,n-2\rangle$ and with ${\rm q} := (\langle 1,2,\ldots,n-2\rangle\setminus{\rm p})^{\rm R}$. \end{lem}

\begin{proof}  We induce on $n$. Note that $\langle(0\,1),(1\,2)\rangle$ and $\langle(1\,2),(0\,1)\rangle$ are the only rearrangements of $\tau(3)$, that $(0\,1)\circ(1\,2)=(0\,2\,1) = (0\,{\rm p}\,2\,{\rm q})$ 
with \ p \ the empty sequence and reverse-complementary to \ q $=\langle 1\rangle$ \ in the number sequence $\langle 1\rangle$, and similarly that $(1\,0)\circ(0\,1) = (0\,1\,2) = (0\,{\rm p}\, 2\,{\rm q})$ 
where \ p $=\langle1\rangle$ and \ q $=\emptyset$.

Choose an integer $k\ge3$. Suppose the lemma holds for $n=k$. Let \ {\bf g} \ be a rearrangement of $\tau(k+1)$. Let ${\bf g'} := {\bf g}\setminus\langle(k-1\,\,k)\rangle$. Now,  
supp$(t)\,\cap\,$supp$  \big((k-1\,\,k)\big)=\emptyset$ for every term $t$ in ${\bf g'}$ except for $t=(k-2\,\,k-1)$. Hence, one of the following two equalities must hold:
\[{\sf 1.}\qquad\qquad \bigcirc{\bf g} = (k-1\,\,k)\circ\bigcirc{\bf g'}\] \[{\sf 2.}\qquad\qquad \bigcirc{\bf g} = \bigcirc{\bf g'}\circ(k-1\,\,k)\] Equality {\sf 1} holds when $(k-2\,\,k-1)>_{\bf g}(k-1\,\,k)$. Equality {\sf 2} 
holds when  $(k-2\,\,k-1)<_{\bf g}(k-1\,\,k)$.

Since ${\bf g'}$ is a rearrangement of $\tau(k)$, we have by the inductive hypothesis that $\bigcirc{\bf g'} = (0\,\,{\rm p'}\,\,k-1\,\,{\rm q'})$ for some subsequence ${\rm p'}$ of $\tau(k)$, where 
${\rm  q'} = \big(\langle 1,2,\ldots,k-2\rangle\setminus{\rm p'}\big)^{\rm R}$. So, if $(k-2\,\,k-1)>_{\bf g}(k-1\,\,k)$, then \[\bigcirc{\bf g} = (k-1\,\,k)\circ\bigcirc{\bf g'} = 
(k-1\,\,k)\circ(0\,\,{\rm p'}\,\,k-1\,\,{\rm q'}) = (0\,\,{\rm p}\,\,k\,\,{\rm q})\quad\mbox{where}\quad {\rm p} := \langle {\rm p'},\,k-1\rangle\quad\mbox{and}\quad {\rm q := q'}.  \]

Similarly, in the event that $(k-2\,\,k-1)<_{\bf g}(k-1\,\,k)$, we have instead that \[\bigcirc{\bf g} = \bigcirc{\bf g'}\circ(k-1\,\,k) = (0\,\,{\rm p'}\,\,k-1\,\,{\rm q'})\circ(k-1\,\,k) = 
(0\,\,{\rm p}\,\,k\,\,{\rm q})\quad\mbox{where}\quad {\rm p:=p'}\quad\mbox{and}\quad {\rm q} := \langle k-1,{\rm q'}\rangle. \]
\noindent These equalities are exactly what the lemma claims. \end{proof}
\vspace{.5em}

Recall our basic sequence $\sigma(n) := \langle\tau(n),(n-1\,\,0)\rangle = \langle(0\,1),(1\,2),\ldots,(n-2\,\,n-1),(n-1\,\,0)\rangle$ for $n\ge3$.

\begin{lem} \label{Numbers2}Let ${\bf f}\in{\rm Seq}(\sigma(n))$ with $n\ge3$. Then $\bigcirc{\bf f} = ({\rm h})(\nu(n)\setminus {\rm h})^-$ for a subsequence ${\rm h}\not=\emptyset$  of $\nu(n)$.  
\end{lem}

\begin{proof}  \underline{Case 1}: \ $(0\,1)<_{\bf f}(n-1\,\,0)$.  Let $m$ be the smallest integer such that 
\[(n-0\,\,0) <_{\bf f} (n-1\,\,n-2) <_{\bf f} (n-2\,\,n-3) <_{\bf f}\cdots <_{\bf f} (m+1\,\,m). \quad\mbox{We can decompose}\quad {\bf f}\quad\mbox{as follows:} \]
\[{\bf f} = \langle {\bf b}_0,(n-1\,\,0), {\bf b}_1,(n-1\,\,n-2), {\bf b}_2, (n-2\,\,n-3),\ldots, {\bf b}_{n-m-1}, (m+1\,\,m), {\bf b}_{n-m}\rangle.\] 
\[\mbox{Then}\quad \bigcirc{\bf f} = \bigcirc{\bf b}_0\circ(n-1\,\,0)\circ\bigcirc{\bf b}_1\circ (n-1\,\,n-2)\circ\bigcirc{\bf b}_2\circ\cdots\circ\bigcirc{\bf b}_{n-m-1}\circ(m+1\,\,m)\circ\bigcirc{\bf b}_{n-m}.\] 
Since ${\bf f}\in{\rm Seq}(\sigma(n))$, there are exactly two terms in ${\bf f}$ whose supports contain $n-1$; those two terms are $(n-1\,\,0)$ and $(n-1\,\,n-2)$. Since those terms border the transpositional 
sequence ${\bf b}_1$, they do not occur as terms in ${\bf b}_1$. Consequently $n-1\not\in{\rm supp}(\bigcirc{\bf b}_1)$. By hypothesis $(0\,1) <_{\bf f} (n-1\,\,0)$, and hence $(0\,1) <_{\bf f} {\bf b}_1$.  Thus 
neither of the two terms of ${\bf f}$ which have $0$ in their supports are terms in ${\bf b}_1$. Therefore ${\rm supp}(\bigcirc{\bf b}_1)\cap{\rm supp}\big((n-1\,\,0)\big) = \emptyset$. So 
$(n-1\,\,0)\circ\bigcirc{\bf b}_1 = \bigcirc{\bf b}_1\circ(n-1\,\,0)$. Thus we infer that 
\[\bigcirc{\bf f} = \bigcirc{\bf b}_0\circ\bigcirc{\bf b}_1\circ (n-1\,\,0)\circ(n-2\,\,n-1)\circ \bigcirc{\bf b}_2\circ\cdots\circ\bigcirc{\bf b}_{n-m-1}\circ(m+1\,\,m)\circ\bigcirc{\bf b}_{n-m}.\] Similarly we see 
that ${\rm supp}(\bigcirc{\bf b}_2)\cap{\rm supp}\big((n-1\,\,0)\circ(n-1\,\,n-2)\big) = \emptyset$, and thus that 
\[\bigcirc{\bf f} = \bigcirc{\bf b}_0\circ\bigcirc{\bf b}_1\circ\bigcirc{\bf b}_2\circ(n-1\,\,0)\circ(n-1\,\,n-2)\circ(n-2\,\,n-3)\circ\cdots\circ\bigcirc{\bf b}_{n-m-1}\circ(m+1\,\,m)\circ\bigcirc{\bf b}_{n-m}.\] 
Continuing in this fashion, we eventually obtain that 
\[\bigcirc{\bf f} = \bigcirc{\bf b}_0\circ\bigcirc{\bf b}_1\circ\cdots\circ\bigcirc{\bf b}_{n-m}\circ(n-1\,\,0)\circ(n-1\,\,n-2)\circ(n-2\,\,n-3)\circ\cdots\circ(m+1\,\,m).\]

Define ${\bf g} := \langle{\bf b}_0,{\bf b}_1,\ldots,{\bf b}_{n-m-1},{\bf b}_{n-m}\rangle$. Note that ${\bf g} = {\bf f}\setminus\langle(n-1\,\,0),(n-1\,\,n-2),\ldots,(m+1\,\,m)\rangle$. Since 
${\bf f}\in{\rm Seq}(\sigma(n))$, we see that ${\bf g}\in{\rm Seq}(\tau(m+1))$, recalling that $\tau(m+1) = \langle(0\,1),(1\,2),\ldots(m-1\,\,m) \rangle$. So by Lemma \ref{Numbers1}, we have that 
$\bigcirc{\bf g} = (0\,\,{\rm p}\,\,m\,\,{\rm q})$, where ${\rm p}$ is a subsequence of $\langle 1,2,\ldots,m-1\rangle$ and where ${\rm q} = (\langle 1,2,\ldots,m-1\rangle\setminus{\rm p})^{\rm R}$. 
Thus $\bigcirc{\bf f} = \bigcirc{\bf g}\circ(n-1\,\,0)\circ(n-1\,\,n-2)\circ(n-2\,\,n-3)\circ\cdots\circ(m+1\,\,m) = (0\,\,{\rm p}\,\,m\,\,{\rm q})\circ(n-1\,\,0)\circ(n-1\,\,n-2)\circ\cdots\circ(m+1\,\,m) = 
(0\,\,{\rm p}\,\,m+1\,\,m+2\,\,\ldots\,n-2\,\,n-1)({\rm q}\,\,m)$. Setting ${\rm h} := \langle 0\,\,{\rm p}\,\,m+1\,\,m+2\,\,\ldots\,n-2\,\,n-1\rangle$, we see that ${\rm h}$ is a nonempty subsequence of 
$\nu(n)$ and observe that $\langle{\rm q},m\rangle = (\nu(n)\setminus{\rm h})^{\rm R}$, and that therefore $({\rm q}\,\,m) = (\nu(n)\setminus{\rm h})^-$. So 
$\bigcirc{\bf f} = ({\rm h})(\nu(n)\setminus{\rm h})^-$ as alleged.\vspace{.5em}

\underline{Case 2}: \ $(0\,1)>_{\bf f}(n-1\,\,0)$. Since the argument parallels that for Case 1, we omit it.      \end{proof}  

\begin{lem}\label{Numbers3}  Let $n\ge3$. Let ${\rm h}$ be a proper nonempty subsequence of $\nu(n)$. Then $\bigcirc{\bf f}=({\rm h})(\nu(n)\setminus{\rm h})^-$ for some 
${\bf f}\in{\rm Seq}(\sigma(n))$ .   \end{lem} 

\begin{proof} If the lemma holds for ${\rm h}$ then it holds also for its complement $\nu(n)\setminus{\rm h}$ in $\nu(n)$. For, if $\bigcirc{\bf f} = ({\rm h})(\nu(n)\setminus{\rm h})^-$, 
then $\bigcirc{\bf f}^{\rm R} = (\bigcirc{\bf f})^- = (\nu(n)\setminus{\rm h})({\rm h})^- = (\nu(n)\setminus{\rm h})(\nu(n)\setminus(\nu(n)\setminus{\rm h}))^-$. We induce on $n\ge3$.\vspace{.5em} 

{\sf Basis Step.} There are the six nonempty proper subsequences of $\nu(3)$; they are \[{\rm u}:=\langle0\rangle,\quad {\rm v}:=\langle1\rangle,\quad {\rm w}:=\langle2\rangle, 
\quad {\rm x}:=\langle0,1\rangle,\quad {\rm y}:=\langle0,2\rangle,\quad {\rm z}:=\langle1,2\rangle. \] We use the fact noted in the preceding paragraph. If ${\bf f}_{\rm u} := 
\langle(0\,1),(1\,2),(2\,0)\rangle$ then $\bigcirc{\bf f}_{\rm u} = (0)(2\,1) = ({\rm u})(\nu(3)\setminus{\rm u})^- = ({\rm z})(\nu(3)\setminus{\rm z})^-$. If ${\bf f}_{\rm v} :=
 \langle(0\,1),(2\,0),(1\,2)\rangle$ then $\bigcirc{\bf f}_{\rm v} = (0\,2)(1) = ({\rm v})(\nu(3)\setminus{\rm v})^- = ({\rm y})(\nu(3)\setminus{\rm y})^-$. If ${\bf f}_{\rm w} := 
\langle(1\,2),(0\,1),(2\,0)\rangle$ then $\bigcirc{\bf f}_{\rm w} = (0\,1)(2) = ({\rm w})(\nu(3)\setminus{\rm w})^- = ({\rm x})(\nu(3)\setminus{\rm x})^-$.\vspace{.5em} 
T
he arbitrary length-one sequence $\langle x\rangle$ in $\nu(k+1)$ is a special case. Choose the transpositional sequence ${\bf f}\in{\rm Seq}(\sigma(k+1))$  to be 
\[{\bf f} := \langle(x\,\,x+1),(x+1\,\,x+2),\ldots,(k-1\,\,k),(k\,0),(0\,1),(1\,2),\ldots,(x-2\,\,x-1),(x-1\,\,x)\rangle.\] Then\footnote{modulo $k+1$ of course}  
$\bigcirc{\bf f} = (x)(k\,\,k-1\,\,k-2\,\ldots\,x+1\,\,x-1\,\,x-2\,\dots\,2\,\,1\,\,0) = (x)(\nu(k+1)\setminus\langle x\rangle)^-$, as desired.\vspace{.5em}

Let ${\rm h} := \langle x_1,x_2,\cdots,x_s\rangle$ be a subsequence of $\nu(k+1)$, and let ${\rm h'} := \nu(k+1)\setminus{\rm h} = \langle y_1,y_2,\ldots,y_t\rangle$ be  the complement in $\nu(k+1)$ of ${\rm h}$.
 Of course $s+t=k+1$. 

By the first paragraph in this proof, we can take it both that $x_s=k$. and also that there exist $s$ disjoint subsequences\footnote{some of which may be vacuous} 
${\rm a}_i$ of $\nu(k+1) = \langle{\rm a}_1,x_1,{\rm a}_2,x_2,\ldots,{\rm a}_s,x_s\rangle$. Indeed, ${\rm h'} = {\rm a}_1{\rm a}_2\ldots{\rm a}_s$, where ${\rm h'}$ is expressed here as the 
concatenation of the subsequences ${\rm a}_i$. The ${\bf f}\in{\rm Seq}(\sigma(k+1))$, whose existence this lemma alleges, must satisfy $\bigcirc{\bf f} = 
(x_1\,x_2\,\ldots\,x_s)(y_{k+1-s}\,y_{k-s}\,\ldots\,y_2\,y_1) = ({\rm h})({\rm a_s^Ra_{s-1}^R\cdots a_2^Ra_1^R})$.\vspace{.5em}

Since we have already dealt with the length-one case ${\rm h} = \langle x\rangle$, we now take it that  $2\le|{\rm h}|= s \le k-1$. Recall that $x_s=k$ is the right-most term in the subsequence ${\rm h}$. Let 
${\rm h''} := {\rm h}\setminus\langle x_s\rangle = \langle x_1,x_2,\ldots,x_{s-1}\rangle$. Since therefore $y_t<k$, we have that ${\rm h'} = \nu(k)\setminus{\rm h''} = \langle y_1,y_2,\ldots,y_t\rangle$ is the 
complement\footnote{as well as remaining the complement in $\nu(k+1)$ of ${\rm h}$} in $\nu(k)$ of the sequence ${\rm h''}$. Hence, by the inductive hypothesis, there exists 
${\bf g} \in {\rm Seq}(\sigma(k))$ for which $\bigcirc{\bf g} = ({\rm h'')(h'})^- = (x_1\,\ldots\,x_{s-1})(y_t\,\,y_{t-1}\,\ldots\,y_2\,\, y_1)$. 

We create ${\bf f}\in{\rm Seq}(\sigma(k+1))$ from ${\bf g}\in{\rm Seq}(\sigma(k))$ by replacing the term $(k-1\,\,0)$ of ${\bf g}$ with the sequence $\langle(k\,\,0),(k-1\,\,k)\rangle$. Notice that, whereas 
$(k-1)\bigcirc{\bf g}=0$, we have instead $(k-1)\bigcirc{\bf f} = k = x_s$ and $(k)\bigcirc{\bf f} = 0$. But for all $z\in (k+1)\setminus\{k-1,k\}$ we have $(z)\bigcirc{\bf g} = (z)\bigcirc{\bf f}$. Obviously 
${\rm h} = \langle{\rm h''},x_s\rangle$, and $\bigcirc{\bf f} = ({\rm h})({\rm h'})^-$. \end{proof} 

\begin{thm}\label{NixCircuits} Let ${\cal T}({\bf f})$ have a circuit, where ${\bf f}$ is a sequence in $1^{k-2}2^1$ with $k\ge4$. Then ${\bf f}$ is not {\rm CI}. \end{thm}

\begin{proof} Theorem \ref{Four} establishes this theorem where ${\cal T}({\bf f})$ contains a triangular subgraph. So for $k\ge n\ge4$, let ${\cal T}(\sigma(n))$ a subgraph of ${\cal T}(k)$. Then some 
${\bf g}\in{\rm Seq}(\sigma(n))$ is a subsequence of ${\bf f}$. Let ${\sf D}$ be the family of cyclic components of the permutation $\bigcirc({\bf f\setminus g})$, let ${\sf U} := \{C: n\,\cap\,{\rm supp}(C) = 
\emptyset\}$, and let ${\sf W :=  D\setminus U}$. Let ${\sf u}$ be the permutation whose family of components is ${\sf U}$, and let ${\sf w}$ be the permutation whose family of components is ${\bf W}$. Then 
$\bigcirc({\bf f\setminus g}) = {\sf uw}$. 

The theorem will be proved when we exhibit rearrangements $\{{\bf p},{\bf q}\}\subseteq{\rm Seq}({\bf g})$ such that $\bigcirc\langle{\bf p,f\setminus g}\rangle = \bigcirc{\bf p}\circ{\sf uw}\not\simeq
\bigcirc{\bf q}\circ{\sf uw} = \bigcirc\langle{\bf q,f\setminus g}\rangle$. Moreover, since ${\rm supp}({\sf u})\,\cap\,{\rm supp}(\bigcirc{\bf p}\circ{\sf w}) = \emptyset = 
{\rm supp}({\sf u})\,\cap\,{\rm supp}(\bigcirc{\bf q}\circ{\sf w})$, it will suffice to insist only that $\bigcirc{\bf p}\circ{\sf w} \not\simeq \bigcirc{\bf q}\circ{\sf w}$. There are three cases. \vspace{.7em}

{\sf Case One:} \ $|n\,\cap\,{\rm supp}(C)| = 1$ for each cycle $C\in{\bf W}$.  

We write ${\sf w} = (0\,\,{\rm s}_0)(1\,\,{\rm s}_1)\ldots(n-1\,\,{\rm s}_{n-1})$, where the ${\rm s}_i$ are finite sequences\footnote{which are not required to be nonempty} in $k\setminus n$. By Lemma 
\ref{Numbers3}, for each $i\in n$ there exists ${\bf p}^{(i)}\in{\rm Seq}({\bf g})$ such that $\bigcirc{\bf p}^{(i)} = (i)(0\,\,1\,\,2\ldots\,\,i-2\,\,i-1\,\,i+1\,\,i+2\ldots n-2\,\, n-1)$. Hence 
$\bigcirc{\bf p}^{(i)}\circ{\sf w} = 
(i\,\,{\rm s}_i)(0\,\,{\rm s}_1\,\,1\,\,{\rm s}_2\,\,2\,\,s_3\ldots\,{\rm s}_{i-2}\,\,i-2\,\,{\rm s}_{i-1}\,\,i-1\,\,{\rm s}_{i+1}\,\,i+1\,\,{\rm s}_{i+2}\,\, i+2\ldots n-2\,\,{\rm s}_{n-1}\,\,n-1\,\,{\rm s}_0)$ 
for each $i\in n$. Now pretend that $\bigcirc{\bf p}^{(i)}\circ{\sf w} \simeq \bigcirc{\bf p}^{(0)}\circ{\sf w}$ for all $i\in n$.  Then all $n$ of the cycle-length multisets $K[{\bf p}^{(i)}]$ of these 
permutations  $\bigcirc{\bf p}^{(i)}\circ{\sf w}$ must be identical. Specifying the $K[{\bf p}^{(i)}]$ for each $i\in n$, we see that \[K[{\bf p}^{(i)}] := \{1+|{\rm s}_i|,n-1+\sum\{|{\rm s}_j|: i\not=j\in n\}.\]
Obviously $1+|{\rm s}_i| < n-1+\sum\{|{\rm s}_j|: t\not= j\in n\}$ whenever $i\not= t$. So our assumption that all of the $K[{\bf p}^{(i)}]$ are identical implies that $|{\rm s}_i| = |{\rm s}_0|$ for all $i\in n$.

Again invoking Lemma \ref{Numbers3}, we can find ${\bf q}\in{\rm Seq}({\bf g})$ for which $\bigcirc{\bf q} = (1\,\,0)(2\,\,3\ldots n-2\,\,n-1)$. Then $\bigcirc{\bf q}\circ{\sf w} = 
(0\,\,{\rm s}_1\,\,1\,\,{\rm s}_0)(2\,\,{\rm s}_3\,\,4\,\,{\rm s}_5\ldots n-2\,\,{\rm s}_{n-1}\,\,n-1\,\,{\rm s}_2)$, and so $K[{\bf q}] = \{2+|{\rm s}_1|+|{\rm s}_0|,n-2+\sum_{j=2}^{n-1}|{\rm s}_j|\}$.
Under the assumption that ${\bf f}$ is CI, we must have that $\bigcirc{\bf p}^{(i)}\circ{\sf w} = \bigcirc{\bf q}\circ{\sf w}$ for all $i\in n$, whereupon $K[{\bf p}^{(i)}] = K[{\bf q}]$. Since all of the integers 
$|{\rm s}_i|$ were found to be equal, for $i=0$ we must infer that  
$\{1+|{\rm s}_0|,(n-1)+(n-1)\cdot|{\rm s}_0|\} = \{2+2\cdot|{\rm s}_0|,n-2+(n-2)\cdot|{\rm s}_0|\}$, an impossibility since $\{1,n-1\}\cap\{2,n-2\} = \emptyset$ when $n\ge 4$. So the Theorem holds 
in the Case-One situation.\vspace{.5em} 

The next Case requires an ancillary fact.

\begin{claim}\label{1stClaim} If $n\ge4$ and if ${\bf W}$ contains a cycle $C$ with $|n\,\cap\,{\rm supp}(C)|\ge2$ then there exists ${\bf p}\in{\rm Seq}(\sigma(n))$ such that the permutation 
$(\bigcirc{\bf p}\circ{\sf w})\restrict(n\,\cup\,{\rm supp}({\sf w}))$ has at least three cycles. \end{claim}

\noindent{\it Proof of Claim.} We first suppose that there exists $C := (x\,\,{\rm s}_x\,\,y\,\,{\rm s}_y)\in{\bf W}$ with $n\cap{\rm supp}(C) = \{x,y\}$.  Without loss of generality, we take it that 
$\langle x,y\rangle$ is a subsequence of $\nu(n)$, and we invoke Lemma \ref{Numbers3} to find some ${\bf p}\in{\rm Seq}(\sigma(n))$ for which $\bigcirc{\bf p} = (x\,\,y)(\nu(n)\setminus\langle x,y\rangle)^-$.  

Let ${\sf w'}$ be the permutation whose family of cyclic components is ${\bf W}\setminus\{C\}$. Since two permutations commute if their supports are disjoint,\footnote{We may write ${\sf a\circ b}$ as ${\sf ab}$ 
in order to emphasize that ${\rm supp}({\sf a})\cap{\rm supp}({\sf b}) = \emptyset$.} we compute:  $\bigcirc{\bf p}\circ{\sf w} = \bigcirc{\bf p}\circ{\sf w'}C = 
(x\,\,y)(\nu(n)\setminus\langle x,y\rangle)^-\circ{\sf w'}C = (\nu(n)\setminus\langle x,y\rangle)^-\circ{\sf w'} (x\,\,y)\circ C = 
[(\nu(n)\setminus\langle x,y\rangle)^-\circ{\sf w'}] [(x\,\,y)\circ(x\,\,{\rm s}_x\,\,y\,\,{\rm s}_y)] = [(\nu(n)\setminus\langle x,y\rangle)^-\circ{\sf w'}]\,(x\,\,{\rm s}_y)(y\,\,{\rm s}_x)$. It is thus clear that 
here the permutation $(\bigcirc{\bf p}\circ{\sf w})\restrict(n\cup{\rm supp}({\sf w}))$ has at least three cyclic components. 

More generally, now, suppose there exists $C := (x\,\,{\rm s}_x\,y\,\,{\rm s}_y\,z\,\,{\rm s}_z) \in {\bf W}$ where ${\rm s}_x{\rm s}_y{\rm s}_z$ is an injective sequence in the set $k\setminus n$. Let 
${\rm q} = \langle m_1,m_2,m_3\rangle$ be a rearrangement of the number sequence $\langle x,y,z\rangle$ for which $m_1<m_2<m_3$. Surely either $({\rm q}) = (z\,\,y\,\,x)$ or $({\rm q})^- = (z\,\,y\,\,x)$.

For $({\rm q}) := (z\,\,y\,\,x)$, by  Lemma \ref{Numbers3} there exists ${\bf p}\in{\rm Seq}(\sigma(n))$ with $\bigcirc{\bf p} = ({\rm q})(\nu(n)\setminus{\rm q})^-$. Again, let ${\sf w'}$ be the permutation 
whose family of components is ${\bf W}\setminus\{C\}$. Then $\bigcirc{\bf p}\circ{\sf w} = \bigcirc{\bf p}\circ{\sf w'}C = ({\rm q})({\rm q'})^-{\sf w'}C$, where ${\rm q'} := \nu(n)\setminus{\rm q}$. Thus 
$\bigcirc{\bf p}\circ{\sf w} = [({\sf q'})\circ{\sf w'}][(z\,\,y\,\,x)\circ C] = [({\sf q'})\circ{\sf w'}][(z\,\,y\,\,x)\circ(x\,\,{\rm s}_x\,y\,\,{\rm s}_y\,z\,\,{\rm s}_z)] = 
[({\sf q'})\circ{\sf w'}](x\,{\rm s}_z)(y\,{\rm s}_x)(z\,{\rm s}_y)$ for three number sequences ${\rm s}_t$. So $\bigcirc{\bf p}\circ{\sf w}$ has at least three cycles. 
Thus the claim holds for $({\rm q}) := (z\,\,y\,\,x)$. On the other hand, if $({\rm q})^- := (z\,\,y\,\,x)$, then Lemma \ref{Numbers3} provides a ${\bf p}_1\in{\rm Seq}(\sigma(n))$ for which 
$\bigcirc{\bf p}_1 = (\nu(n)\setminus{\rm q})({\rm q})^-$, and we omit the repetitive rest of the argument. Claim 1 follows.\vspace{.7em}

{\sf Case Two:} \ The family ${\bf W}$ of cycles contains exactly one element $C$, and $n\subseteq{\rm supp}(C)$.

Pick an  integer $i$ with $0\le i, i+1, i+2 < k$. The cycle $C$ is expressable in one of these two ways:

\ \ \ \ \ \underline{Order 1}. \  \ $C = (i\,\,\,{\rm s}_i\,\,\,i+1\,\,\,{\rm s}_{i+1}\,\,\,i+2\,\,\,{\rm s}_{i+2})$

\ \ \ \ \ \underline{Order 2}. \  \ $C = (i\,\,\,{\rm s}_i\,\,\,i+2\,\,\,{\rm s}_{i+2}\,\,\,i+1\,\,\,{\rm s}_{i+1})$

\noindent For the subsequence ${\bf a}_i := \langle(i+1\,\,i+2),(i\,\,i+1)\rangle$ of $\sigma(n)^{\rm R}\in{\rm Seq}(\sigma(n))$, if Order 1 prevails then \[\bigcirc{\bf a}_i\circ C =
(i\,\,i+1\,\,i+2)\circ(i\,\,\,{\rm s}_i\,\,\,i+1\,\,\,{\rm s}_{i+1}\,\,\,i+2\,\,\,{\rm s}_{i+2}) = (i\,\,\,{\rm s}_{i+1}\,\,\,i+2\,\,\,{\rm s}_i\,\,\,i+1\,\,\,{\rm s}_{i+2}).\] Thus, if $C$ is of the form in Order 1, then 
$\bigcirc{\bf a}_i\circ C$ is a single cycle of the same length as that of $C$. But if, instead, Order 2 prevails, then 
\[\bigcirc{\bf a}_i^{\rm R}\circ C = (i\,\,i+2\,\,i+1)\circ(i\,\,\,{\rm s}_i\,\,\,i+2\,\,\,{\rm s}_{i+2}\,\,\,i+1\,\,\,{\rm s}_{i+1}) = (i\,\,\,{\rm s}_{i+2}\,\,\,i+1\,\,\,{\rm s}_i\,\,\,i+2\,\,\,{\rm s}_{i+1}). \] 
So here too, when $C$ is of the form Order 2, then $\bigcirc{\bf a}_i^{\rm R}\circ C$ is a single cycle whose length is $|C|$.\vspace{.5em}

{\sf Subcase}: $n$ is even. Then the transpositional sequence $\sigma(n)$ has an even number of terms. So we may write $\sigma(n)$ as a sequence ${\bf v}_1{\bf v}_2\ldots{\bf v}_t$ of $t := n/2$ pairs 
${\bf v}_i:= \langle(2i-2\,\,\,2i-1),(2i-1\,\,\,2i)\rangle$ of transpositions that are adjacent and consecutive in $\sigma(n)$. Define ${\bf v}'_t := {\bf v}_t^{\rm R}$ if $\langle 2t-2,2t-1,0\rangle = \langle n-2,n-1,0\rangle$ 
occurs in Order 1 in $C$. For each $i\in[t-1]$ we write ${\bf v}'_i$ as ${\bf v}_i^{\rm R}$ if $\langle 2i-2,2i-1,2i\rangle$ occurs in Order 1 in the cycle 
$\bigcirc{\bf v}'_{i+1}\circ\bigcirc{\bf v}'_{i+2}\circ\cdots\circ\bigcirc{\bf v}'_{t-1}\circ\bigcirc{\bf v}'_t\circ C$. 
In the corresponding Order 2 situation we make the opposite definitions for the ${\bf v}'_i$; that is, ${\bf v}'_i := {\bf v}_i$ for each $i\in[t]$. As a consequence of our observations prior to the present Subcase, the 
permutation $\bigcirc{\bf v}'_1\circ\bigcirc{\bf v}'_2\circ\cdots\circ\bigcirc{\bf v}'_t\circ C$ is a single cycle whose length is $|C|$. 

Of course ${\bf p} := {\bf v}'_1{\bf v}'_2\ldots{\bf v}'_t \in {\rm Seq}(\sigma(n))$. Claim 1 tells us that there exists ${\bf q}\in{\rm Seq}(\sigma(n))$ for which $\bigcirc{\bf q}\circ C$ has at least three 
component cycles. So $\bigcirc{\bf p}\circ C \not\simeq \bigcirc{\bf q}\circ C$. \vspace{.3em} 

{\sf Subcase}: $n$ is odd. This time $t := (n-1)/2$ and, for $1\le i\le t$, we define the ${\bf v_i}$ and the ${\bf v}'_i$ as above. Here, $\sigma(n) = \langle{\bf v}_1,{\bf v}_2,\ldots,{\bf v}_t,(n-1\,\,0)\rangle$. 
Now let  ${\bf p} := \langle{\bf p'},(n-1\,\,\,0)\rangle := \langle{\bf v}'_1,{\bf v}'_2,\ldots,{\bf v}'_t,(n-1\,\,\,0)\rangle$. Then ${\bf p} \in {\rm Seq}(\sigma(n))$. So $\bigcirc{\bf p}\circ C =  
[\bigcirc{\bf p'}\circ C]\circ(n-1\,\,\,0)$ is the product of the cycles $\bigcirc{\bf p'}\circ C$ and $(n-1\,\,\,0)$. Moreover, $\{n-1,0\}\subset{\rm supp}(\bigcirc{\bf p}\circ C)$. Therefore $\bigcirc{\bf p}\circ C$ 
has exactly two cyclic components. Claim 1 promises us a ${\bf q}\in{\rm Seq}(\sigma(n))$ for which $\bigcirc{\bf q}\circ C$ has more than two cyclic components, whence 
$\bigcirc{\bf q}\circ C \not\simeq \bigcirc{\bf p}\circ C$. So the theorem holds under Case Two circumstances.\vspace{.7em}

{\sf Case Three}: \ ${\bf W}$ contains a cycle $C$ for which $1 < |n\cap{\rm supp}(C)| < n$. As before, let ${\sf w}\in{\rm Sym}(k)$ be the permutation whose family of cyclic components is ${\sf W}$, 
and let ${\sf w'}$ be the permutation whose family of (nontrivial) cyclic components is ${\bf W}\setminus\{C\}$.    Since $|n\cap{\rm supp}(C)|<n$, there exists $m\in n\setminus{\rm supp}(C)$. If 
$m\not\in{\rm supp}({\sf w'})$, then let $Q$ be the trivial cycle. But if $m\in{\rm supp}({\sf w'})$ then let $Q$ be the unique cycle in ${\bf W}\setminus\{C\}$ such that $m\in{\rm supp}(Q)$, and let 
${\sf w''}$ be the permutation whose family of nontrivial cyclic components is ${\bf W}\setminus\{C,Q\}$. Let ${\rm h}$ be the subsequence of $\nu(n)$ with ${\rm supp}({\rm h}) = n\cap{\rm supp}(Q)$. [If 
$Q=(m)$, we let ${\rm h}:=\langle m\rangle$.]

Since ${\rm h}$ is a nonempty proper subsequence of $\nu(n)$, by Lemma \ref{Numbers3} there exists ${\bf p}\in{\rm Seq}(\sigma(n))$ such that $\bigcirc{\bf p} = ({\rm h})(\nu(n)\setminus{\rm h})^-$. 
Thus $\bigcirc{\bf p}\circ{\sf w} = \bigcirc{\bf p}\circ CQ{\sf w''} = ({\rm h})(\nu(n)\setminus{\rm h})^-\circ CQ{\sf w''} = \big((\nu(n)\setminus{\rm h})^-\circ C{\sf w''}\big)\big(({\rm h})\circ Q\big)$,  which is the 
product of two permutations,  $(\nu(n)\setminus{\rm h})^-\circ C{\sf w''}$ and $({\rm h})\circ Q$, whose supports are disjoint. 

Now, if $\bigcirc{\bf p}\circ{\sf w}\restrict\big(n\cup{\rm supp}({\sf w})\big)$ has fewer than three cyclic components, then we employ Claim 1 to obtain some ${\bf q}\in{\rm Seq}(\sigma(n))$ such that 
$\bigcirc{\bf q}\circ{\sf w}\restrict\big(n\cup{\rm supp}({\sf w})\big)$ has at least three cycles, whence $\bigcirc{\bf q}\circ{\sf w} \not\simeq \bigcirc{\bf p}\circ{\sf w}$. So it remains only to deal with the situation 
where $\bigcirc{\bf p}\circ{\sf w}\restrict\big(n\cup{\rm supp}({\sf w})\big)$ has at least three cycles.

Suppose $\bigcirc{\bf p}\circ{\sf w}\restrict\big(n\cup{\rm supp}({\sf w})\big)$ has at least three cycles. Then there are two possibilities to treat; to wit:

(1). \  $(\nu(n)\setminus{\rm h})^-\circ C{\sf w''}\restrict Y$ has more than one cycle, where $Y := {\rm supp}(C{\sf w''})\cup{\rm supp}\big((\nu(n)\setminus{\rm h})^-\big)$.

(2). \  $({\rm h})\circ Q\restrict X$ has more than one cycle, where $X := {\rm supp}(Q)\cup\big(n\setminus{\rm supp}((\nu(n)\setminus{\rm h})^-)\big)$. \vspace{.7em}

\noindent FIRST POSSIBILITY: The permutation $(\nu(n)\setminus{\rm h})^-\circ C{\sf w''}\restrict Y$ has more than one cycle. Here we need 

\begin{claim} Each orbit of $(\nu(n)\setminus{\rm h})^-\circ C{\sf w''}\restrict Y$  contains at least one element in ${\rm supp}((\nu(n)\setminus{\rm h})^-)$. \end{claim}

\noindent{\it Proof of Claim.} Let $E$ be an orbit of $(\nu(n)\setminus{\rm h})^-\circ C{\sf w''}\restrict Y$. Let $e\in E$. Then $E = \{e\big((\nu(n)\setminus{\rm h})^-\circ C{\sf w''}\big)^i:i\in{\mathbb Z}\}$. 
We are done if $e\in{\rm supp}\big((\nu(n)\setminus{\rm h})^-\big)$. So suppose $e \not\in {\rm supp}\big((\nu(n)\setminus{\rm h})^-\big)$.  Then, since $e\in Y$, it follows that $e\in{\rm supp}(C{\sf w''})$, and hence 
that either $e\in{\rm supp}(C)$ or $e\in{\rm supp}({\sf w''})$. 

First, suppose that $e\in{\rm supp}(C)$. Then $eC^i\in{\rm supp}(C)\subseteq{\rm supp}(C{\sf w''})$ for all $i\in{\mathbb Z}$. Also, there is a least positive integer $l$ with $eC^l\in n\cap{\rm supp}(C)$. Now, 
$n\cap{\rm supp}(C) \subseteq {\rm supp}\big((\nu(n)\setminus{\rm h})^-\big)$. Hence $eC^l \in {\rm supp}((\nu(n)\setminus{\rm h})^-)$. Since $eC^i \not\in {\rm supp}((\nu(n)\setminus{\rm h})^-)$ for all 
$i \in \{0,1,\ldots,l-1\}$, we have that $eC^l = e\big((\nu(n)\setminus{\rm h})^-\circ C{\sf w''}\big)^l \in E$. So $eC^l \in E\,\cap\,{\rm supp}\big((\nu(n)\setminus{\rm h})^-\big)$. Thus   
$E\cap{\rm supp}\big((\nu(n)\setminus{\rm h})^-\big) \not= \emptyset$, as claimed. 

Next, suppose instead that $e \in {\rm supp}({\sf w''})$. Then $e \in {\rm supp}(F)$ for some cycle $F \in {\bf W}\setminus\{C,Q\}$. Since $F \in {\bf W}$, we have that $n\cap{\rm supp}(F) = 
\big({\rm supp}({\rm h})\cup{\rm supp}((\nu(n)\setminus{\rm h})^-)\big)\cap{\rm supp}(F) \not= \emptyset$. Also, since $F \in {\bf W}\setminus\{Q\}$, we have that ${\rm supp}(F)\cap{\rm supp}(Q) = \emptyset$, 
and hence that ${\rm supp}(F)\cap{\rm supp}\big(({\rm h})\big) = \emptyset$ since ${\rm supp}\big(({\rm h})\big) \subseteq {\rm supp}(Q)$. So ${\rm supp}(F)\cap{\rm supp}\big((\nu(n)\setminus{\rm h})^-\big) \not= \emptyset$. This time let $l$ denote the least positive integer such that $eF^l \in {\rm supp}(F)\cap{\rm supp}\big((\nu(n)\setminus{\rm h})^-\big)$. Let ${\sf w'''}$ be the permutation whose family of component cycles is 
${\bf W}\setminus\{C,Q,F\}$. Since $eF^i \not\in {\rm supp}\big((\nu(n)\setminus{\rm h})^-\circ C{\sf w'''}\big)$ when $i \in \{0,1,\ldots,l-1\}$, it follows that $eF^l = e\big((\nu(n)\setminus{\rm h})^-\circ CF{\sf w'''}\big) =
e\big((\nu(n)\setminus{\rm h})^-\circ C{\sf w''}\big)$, whence $eF^l \in E$. But then $eF^l \in E\cap{\rm supp}((\nu(n)\setminus{\rm h})^-)$, and again we have that 
$E\cap{\rm supp}((\nu(n)\setminus{\rm h})^-) \not= \emptyset$. The proof of Claim 2 is complete.\vspace{.7em}

\begin{claim} There exist orbits $A \not= B$ of $(\nu(n)\setminus{\rm h})^-\circ C{\sf w''}\restrict Y$ and elements $x$ and $y$ in ${\rm supp}((\nu(n)\setminus{\rm h})^-)$ with $x \in A$ and $y \in B$, and such that $y(\nu(n)\setminus{\rm h})^- = x$. \end{claim}

\noindent{\it Proof of Claim.} Let $B$ be an orbit of $G\circ C{\sf w''}\restrict Y$, where $G := (\nu(n)\setminus{\rm h})^-$. By Claim 2, there exists $b\in B\cap{\rm supp}(G)$. If $bG^i \in B$ for every $i\in{\mathbb N}$, then
 ${\rm supp}(G)\subseteq B$, contrary to Claim 2, since by hypothesis $G\circ C{\sf w''}\restrict Y$ has at least two orbits.   So $bG^j \in B$ while $bG^{j+1} \not\in B$ for some $j\in{\mathbb N}$. Let $y := bG^j$, let $x := yG$, and let $A$ be the orbit of $G\circ C{\sf w''}\restrict Y$ for which $x\in A$. Claim 3 follows. \vspace{.7em}

Let $x, y, A, B,$ and $G := (\nu(n)\setminus{\rm h})^-$ be as in Claim 3, and let ${\rm d}$ be the subsequence of $\nu(n)$ such that the term set of ${\rm d}$ is $\{x\}\cup{\rm supp}(({\rm h}))$. [If 
${\rm h} = \langle m\rangle$, let the term set of ${\rm d}$ be $\{m,x\}$.] Since $n\cap{\rm supp}(C) \subseteq {\rm supp}(G)$,  and since $|n\cap{\rm supp}(C)| > 1$, it follows that $|{\rm supp}(G)| > 1$. 
Since $x$ is the only term of ${\rm d}$ which belongs to ${\rm supp}(G)$, it follows that there is an element in ${\rm supp}(G)$, and hence in $n$, which is not a term in ${\rm d}$. So ${\rm d}$ is a 
proper subsequence of $\nu(n)$. The number sequence ${\rm d}$ was produced by inserting $x$ as a term into the sequence ${\rm h}$, and so the sequence $\nu(n)\setminus{\rm d}$ is obtained by deleting the term $x$
 from the sequence $\nu(n)\setminus{\rm h}$.  Since $|{\rm d}|\ge2$, there exists $z\in{\rm supp}\big(({\rm d})\big)$ such that $z({\rm d}) = x$. But $z\not= x$, and so $z \in {\rm supp}\big(({\rm h})\big)$; we can write 
$({\rm h}) = (z\,\,{\rm s}_z)$, and so $({\rm d}) = (z\,\,x)\circ ({\rm h}) = (z\,\,x)\circ(z\,\,{\rm s}_z) = (z\,\,x\,\,{\rm s}_z)$. Similarly, since $yG = x$, we may write $G = (y\,\,x\,\,{\rm s}_x)$. Delete the term $x$ from $G$, 
and obtain that $(\nu(n)\setminus{\rm d})^- = (y\,\,{\rm s}_x)$. Thus, $(y\,\,x)\circ G = (y\,\,x)\circ(\nu(n)\setminus{\rm h})^- = (y\,\,x)\circ(y\,\,x\,\,{\rm s}_x) = (x)(y\,\,{\rm s}_x) = (\nu(n)\setminus{\rm d})^-$. These equalities enable us to expand the product: \ $\bigcirc{\bf q}\circ{\sf w} = \bigcirc{\bf q}\circ CQ{\sf w''} = ({\rm d})(\nu(n)\setminus{\rm d})^-\circ CQ{\sf w''} =  (z\,\,x)\circ({\rm h})(\nu(n)\setminus{\rm d})^-\circ CQ{\sf w''} = (z\,\,x)\circ({\rm h})(y\,\,x)\circ(\nu(n)\setminus{\rm h})^-\circ CQ{\sf w''} = (z\,\,x)\circ(y\,\,x)\circ({\rm h})(\nu(n)\setminus{\rm h})^-\circ CQ{\sf w''} = (z\,\,x)\circ(y\,\,x)\circ\bigcirc{\bf p}\circ {\sf w}$.

Recall that $x$ and $y$ are elements in distinct orbits $A$ and $B$ of $G\circ C{\sf w''}\restrict Y$. Recall also that $\bigcirc{\bf p}\circ{\sf w} = [({\rm h})\circ Q][G\circ C{\sf w''}]$. But 
$[{\rm supp}(({\rm h})\circ Q)] \cap [{\rm supp}(G\circ C{\sf w''})] = \emptyset$, and $A$ and $B$ are distinct orbits of $\bigcirc{\bf p}\circ{\sf w}$ as well. Also, since $z \in {\rm supp}(({\rm h})) \subseteq 
{\rm supp}(({\rm h})\circ Q)$, we have that $z$ must belong to a third orbit $D$ of $\bigcirc{\bf p}\circ{\sf w}\restrict(n\cup{\rm supp}({\sf w}))$. Consequently the sets $A, B, C$ will amalgamate to form a 
single orbit $A\cup B\cup C$ of $\bigcirc{\bf p}\circ{\sf w}\restrict(n\cup{\rm supp}({\sf w}))$. Thus the permutation $\bigcirc{\bf q}\circ{\sf w}\restrict(n\cup{\rm supp}({\sf w}))$ will possess exactly two 
fewer orbits than the permutation $\bigcirc{\bf p}\circ{\sf w}\restrict(n\cup{\rm supp}({\sf w}))$. Hence $\bigcirc{\bf q}\circ{\sf w} \not\simeq \bigcirc{\bf p}\circ{\sf w}$. \vspace{.8em}

\noindent SECOND POSSIBILITY: \ $({\rm h})\circ Q$ has more than one orbit. \vspace{.3em} 

If $|n\cap{\rm supp}(Q)| = 1$, then $({\rm h})\circ Q$ is a single cycle, contrary to the present hypothesis. Thus $|{\rm supp}\big(({\rm h})\big)| = |n\cap{\rm supp}(Q)| > 1$. So, arguing as in the First Possibility, 
we can find $\{x,y\}\subseteq{\rm supp}\big(({\rm h})\big)$ and distinct orbits $A$ and $B$ of $({\rm h})\circ Q$ with $\langle x,y\rangle \in A\times B$ and such that $x({\rm h}) = y$. Let ${\rm d}$ be the 
sequence obtained by deleting the term $x$ from the sequence ${\rm h}$. By Lemma \ref{Numbers3}, there exists ${\bf q}\in{\rm Seq}(\sigma(n))$ such that $\bigcirc{\bf q} = ({\rm d})(\nu(n)\setminus{\rm d})^-$. 

Observe that $({\rm d}) = ({\rm h})\circ(x\,\,y)$, and that if an element $z \in {\rm supp}\big((\nu(n)\setminus{\rm d})^-\big)$ satisfies $z(\nu(n)\setminus{\rm d})^- = x$ then $(\nu(n)\setminus{\rm d})^- = 
(z\,\,x)\circ(\nu(n)\setminus{\rm h})^- = (\nu(n)\setminus{\rm h})^-\circ(x\,\,t)$ where $t := x(\nu(n)\setminus{\rm d})^-$. So $\bigcirc{\bf q}\circ{\sf w} = \bigcirc{\bf q}\circ CQ{\sf w''} = 
({\rm d})(\nu(n)\setminus{\rm d})^-\circ CQ{\sf w''} = ({\rm h})\circ(x\,\,y)(\nu(n)\setminus{\rm h})^-\circ(x\,\,t)\circ CQ{\sf w''} = ({\rm h})(\nu(n)\setminus{\rm h})^-\circ CQ{\sf w''}(x\,\,y)\circ(x\,\,t) = 
\bigcirc{\bf q}\circ{\sf w}\circ(x\,\,y)(t\,\,x)$. As in the First Possibility, we encounter $x, y$, and $t$ as elements in distinct orbits $A, B$, and $C$ of the permutation 
$\bigcirc{\bf p}\circ{\sf w}\restrict(n\cup{\rm supp}({\sf w})$. By an argument similar to that in the First Possibility, we infer that $\bigcirc{\bf q}\circ{\sf w} \not\simeq \bigcirc{\bf p}\circ{\sf w}$. So 
${\bf f}$ is not CI in Case Three too, and thus Theorem \ref{NixCircuits} is proved. \end{proof}  

We have completed the proof of Theorem \ref{CCIT}, which tells us exactly which connected transpositional multigraphs are CI. This renders it easy to specify the class of all CI transpositional multigraphs on the 
vertex set $n$. Recall that, where $n\in \{1,2\}$, every transpositional sequence is both permutatially complete and conjugacy invariant. The following summarizes the main results in \S3.

\begin{thm}\label{SummaryCI} For $n\ge 3$, let ${\bf u}$ be a sequence in $1^{n-2}2^1$. Let $\{{\cal T}({\bf u}_i): i\in m\}$ be the set of components of the transpositional multigraph ${\cal T}({\bf u})$,   
where each ${\bf u}_i$ is the subsequence of ${\bf u}$ for which the vertex set of ${\cal T}({\bf u}_i)$ is $V_i = \bigcup{\rm Supp}({\bf u}_i)$, and where of course $\{V_i:i\in m\}$ is a partition of the set $n$. Then: 

${\bf u}$ is conjugacy invariant if and only if ${\bf u}_i$ is conjugacy invariant for every $i\in m$.

${\bf u}_i$ is {\rm CI} for $|V_i|=3$ if and only if either $|{\bf u}_i|$ is odd or ${\cal T}({\bf u}_i)$ is a multitree with a simple multitwig. 

${\bf u}_i$ is {\rm CI} for $|V_i|\ge4$ if and only if ${\cal T}({\bf u}_i)$ is a multitree, no vertex of which is on more than one nonsimple multiedge, and each even-multiplicity multiedge of which is a multitwig 
whose non-leaf vertex has exactly two neighbors. \end{thm}

\begin{proof}  The theorem's first claim is obvious. Its second claim is immediate from Theorem \ref{n=3}. Its third claim merely combines Theorems \ref{SufficiencyCI}, \ref{NecessaryCI}, \ref{Four} and 
\ref{NixCircuits}. \end{proof}

\subsection{Unfinished work}

If a sequence ${\bf s}$ in ${\rm Sym}(n)$ is perm-complete then of course ${\rm Prod}({\bf s})$ is a coset of the subgroup ${\rm Alt}(n)$ of ${\rm Sym}(n)$. Ross Willard asks for what other ${\bf s}$ are there 
subgroups $H_{\bf s} < {\rm Sym}(n)$ for which ${\rm Prod}({\bf s})\in{\rm Sym}(n)/H_{\bf s}$. 

An obvious task ahead pertaining to conjugacy invariance is the formidible one of providing necessary and sufficient criteria for deciding conjugacy invariance of every permutational sequence ${\bf s}$ in 
${\rm Sym}(n)$. The ultimate goal is criteria enabling one to recognize the family  ${\cal C}_{\bf s}$ of conjugacy classes ${\sf C}$ of ${\rm Sym}(n)$ for which ${\sf C}\cap{\rm Prod}({\bf s})\not=\emptyset$.

Every $f\in{\rm Sym}(n)$ has an infinite number of factorizations into products of transpositions. But if the lengths of the nontrivial cyclic components of $f$ are $\ell_1,\ell_2,\ldots,\ell_d$ then the length of every minimal 
${\bf t}$ is $\big(\sum_{i=1}^d \ell_i\big)-d$, and $f$ has, by our definition given now, exactly $\Phi({\rm Type}(f)) > \sum_{i=1}^d\ell_i^{\ell_i-2}$ distinct minimal length transpositional factorizations if $|{\rm supp}(f)|\ge5$.
\vspace{.5em}

\noindent{\bf Problem.} \ Specify exact values for $\Phi({\rm Type}(f))$. The enumeration gets nontrivial when $f$ is not single-cycled.\vspace{.5em}

Clearly, if every term $s_i$ of the sequence ${\bf s} := \langle s_0,s_1,\ldots,s_m\rangle$ in ${\rm Sym}(n)$ 
has a factorization, $s_i = \bigcirc{\bf t}_i = t_{i,0}\circ t_{i,1}\circ\cdots\circ t_{i,l_i}$ into a product of transpositions such that the conglomerate transpositional sequence ${\bf t} := {\bf t}_0{\bf t}_1\cdots{\bf t}_m$ 
is conjugacy invariant, then the permutational sequence ${\bf s}$ itself is conjugacy invariant. Thus we quickly get a sufficient condition for ${\bf s}$ to be conjugacy invariant. However, that condition is not  
necessary to assure the conjugacy invariance of a permutational sequence.\vspace{.7em}

\noindent{\bf Counterexample.} Let ${\bf s} := \langle(0\,\,1\,\,2),(0\,\,2\,\,1)^{\beta(2)}\rangle$. We omit the easy verification that ${\rm Prod}({\bf s})\subseteq 3^1$, whence ${\bf s}$ is conjugacy invariant. 
However, $(0\,\,1\,\,2)$ has exactly three distinct factorizations as a product of two transpositions; these are: 
\[(0\,\,1\,\,2) = (0\,\,1)\circ(0\,\,2)\qquad\qquad (0\,\,1\,\,2) = (0\,\,2)\circ(1\,\,2)\qquad\qquad (0\,\,1\,\,2) = (1\,\,2)\circ(0\,\,1)\] Of course $(0\,\,2\,\,1)$ likewise has exactly three such factorizations, and since 
the permutation  $(0\,\,2\,\,1)$ occurs exactly twice as a term in ${\bf s}$, we infer each sequence ${\bf t}$ in $1^12^1$ that results from factorizations of each term of ${\bf s}$ into products of two transpositions  
per term is six terms long. 

The reader can check that there are exactly three distinct ${\rm Seq}({\bf t}_i)$ that result from the possible length-$6$ conglomerate transpositional sequences. As usual, each such 
${\rm Seq}({\bf t}_i)$, for $i\in3$, determines a transpositional multigraph ${\cal T}({\bf t}_i)$ on the vertex set $3$. We list the multiedge sets of these three multigraphs; they are: 
\[E_0 := \{(0\,\,1)^{\beta(2)},(1\,\,2)^{\beta(2)},(0\,\,2)^{\beta(2)}\}\qquad E_1 =\{(0\,\,1)^{\beta(2)},(1\,\,2),(0\,\,2)^{\beta(3)}\}\qquad E_2 = \{(0\,\,1)^{\beta(3)},(1\,\,2)^{\beta(3)}\}\]  By Theorem 
\ref{SummaryCI}, none of these three transpositional multigraphs ${\cal T}({\bf t}_i)$ is conjugacy invariant.\vspace{.5em}

This counterexample exhibits a conjugacy invariant permutational sequence ${\bf s}$ which lacks a conjugacy invariant conglomerate transpositional sequence that results from transpositional factorizations 
of the terms in ${\bf s}$. We leave it to the reader to corroborate that the permutational sequence $\langle(0\,\,1\,\,2),(0\,\,3\,\,2),(0\,\,3\,\,1)\rangle$ in ${\rm Sym}(4)$ is a second, perhaps more interesting, 
such counterexample.\vspace{.8em}

A transposition is a special sort of ``single-cycled'' permutation; i.e., an $f \in \bigcup\{1^{n-c}c^1:2\le c\le n\}$. Arthur Tuminaro \cite{Tuminaro} kicked off the study of conjugacy invariance of sequences of single-cycled 
permutations.

\vspace{1em}

\noindent{\bf Key words and expressions.}  permutational/transpositional sequence, transpositional multigraph, compositional product, \ Blt$(n)$, permutationally complete, conjugacy invariant, single-cycled 
permutation.\vspace{1em}

\noindent{\bf 2010 Mathematical Subject Classification.} 20B30, 20B99

\end{document}